\documentclass{amsart}
\usepackage[english]{babel}
\usepackage[utf8x]{inputenc}
\usepackage[T1]{fontenc}

\usepackage{a4wide}				
\usepackage{emptypage}          
\usepackage{amssymb}            
\usepackage{amsmath}            
\allowdisplaybreaks              
\usepackage[scr=dutchcal]{mathalfa}

\usepackage{amsthm}             
\usepackage{graphicx}           
\usepackage{color}              
\usepackage{setspace}           
\usepackage{xifthen}            
\usepackage[makeroom]{cancel}   
\usepackage{mathtools}          
\usepackage{thmtools}
\usepackage{etoolbox}           
\usepackage{hyperref}
\hypersetup{
    colorlinks=true,
    linkcolor=black,
    filecolor=magenta,      
    urlcolor=cyan,
    citecolor=black
    }
\usepackage[hang,flushmargin]{footmisc}     
\usepackage[noabbrev]{cleveref}           
\usepackage{tikz}               
\usetikzlibrary{arrows.meta}    
\usetikzlibrary{decorations.pathreplacing} 
\usetikzlibrary{decorations.markings}
\usepackage{adjustbox}          
\usepackage{changepage}         
\usepackage{dsfont}             
\usepackage{wrapfig}           
\usepackage{float}             
\usepackage{esint}				

\usepackage{amsfonts}
\usepackage{pdfsync}
\usepackage{epsfig}
\usepackage{url}
\usepackage{calligra}

\newcommand{\virgolette}[1]{``#1''}
\usepackage{amssymb}      
\usepackage{eufrak}       
\newcommand{\D}{\ensuremath{\EuFrak{D}}}

\newcommand{\R}{\mathbb{R}}
\newcommand{\N}{\mathbb{N}}
\newcommand{\Z}{\mathbb{Z}}
\newcommand{\T}{\mathbb{T}_T^d}
\newcommand{\ud}{\mathrm{d}}
\newcommand{\F}{\mathcal{F}}
\newcommand{\FF}{\underline{\mathcal{F}_p^0}}
\newcommand{\Ff}{\mathscr{f}_p}
\newcommand{\Fff}{\underline{\mathscr{f}_p}}

\newcommand{\J}{\mathcal{J}}
\newcommand{\I}{\mathcal{I}}
\newcommand{\X}{\mathbb{X}}

\newcommand{\suchthat}{\;\ifnum\currentgrouptype=16 \middle\fi|\;}

\newcommand{\eps}{\varepsilon}

\newcommand{\comment}[1]{}          

\theoremstyle{plain}
\newtheorem{theorem}{Theorem}[section]
\newtheorem{lemma}[theorem]{Lemma}
\newtheorem{proposition}[theorem]{Proposition}

\theoremstyle{definition}
\newtheorem{definition}[theorem]{Definition}
\numberwithin{equation}{section}
\newtheorem{remark}[theorem]{Remark}

\AfterEndEnvironment{definition}{\noindent\ignorespaces}    
\AfterEndEnvironment{proposition}{\noindent\ignorespaces}
\AfterEndEnvironment{remark}{\noindent\ignorespaces}
\AfterEndEnvironment{lemma}{\noindent\ignorespaces}
\AfterEndEnvironment{theorem}{\noindent\ignorespaces}
\AfterEndEnvironment{proof}{\noindent\ignorespaces}

\DeclareRobustCommand{\rchi}{{\mathpalette\irchi\relax}}
\newcommand{\irchi}[2]{\raisebox{\depth}{$#1\chi$}} 

 \title[\(\Gamma\)-Convergence for \((s,p)\)-Gagliardo Energies and minimisation results]%
{\textbf{{\large 
     \(\Gamma\)-convergence, variational analysis and characterisation of minimisers for \\ \lowercase{\((s,p)\)}-Gagliardo energies in the flat \lowercase{\(d\)}-torus}}}

\author[G. Pini, F. Santilli]{G. Pini, F. Santilli}
\date{}

\begin{document}

\maketitle

\begin{abstract}
This paper deals with the variational analysis, for every \(s \in (0,1)\) and \(p \in [1,+\infty)\), of $(s,p)$-Gagliardo seminorms in a periodic setting. First, we consider the space of $L^p$, $T$-periodic functions and define the energy functional $\mathcal{F}_p^s$ as the density of the \(d\)-dimensional $(s,p)$-Gagliardo seminorm over the periodic cell. Our goal is to rigorously characterise the $\Gamma$-limits of this functional as the fractional parameter $s$ approaches its endpoint values, $0^+$ and $1^-$. We prove that, as $s \to 0^+$, the rescaled energy $s\mathcal{F}_p^s$ $\Gamma$-converges to a functional $\mathcal{F}_p^0$ defined by the double integral of $|u(x)-u(y)|^p$ over the periodic cell. Then, for the limit as $s \to 1^-$, we establish that the rescaled energy $(1-s)\mathcal{F}_p^s$ $\Gamma$-converges to the classical Dirichlet $p$-energy, extending known results from bounded domains to the periodic framework.

Finally, we analyse the one-dimensional minimiser of the energy $\mathcal{F}_p^s$ for $s \in (0,1)$ and the limit functional $\mathcal{F}_p^0$ within the special class of piecewise affine periodic functions whose distributional derivative consists of a constant absolutely continuous part and a singular part with opposite sign and quantised jumps. In this setting, the energy depends only on the position of these jump points, and we prove that the absolute minimiser is achieved by their equispaced configuration.

\vskip .3truecm \noindent \textsc{Keywords}: \(\Gamma\)-convergence, Fractional seminorms, Periodic minimisers
\vskip.1truecm \noindent \textsc{Mathematics Subject Classification}: 49J45, 49J10, 46E35.
\end{abstract}

\tableofcontents

\section*{Introduction} \label{sez: introduzione} 
\noindent
The study of non-local energies has become a central theme in the field of Calculus of Variations. A canonical example is the family of $(s,p)$-Gagliardo seminorms \([u]^p_{s,p}\), for \(s \in (0,1)\), \(p \in [1,+\infty)\). The asymptotic behaviour of these seminorms as the fractional parameter $s$ approaches its endpoint values, $0^+$ and $1^-$, has been the subject of intense investigation. A foundational result in this area is due to Bourgain, Brezis and Mironescu \cite{Brezis01, Brezis02}, in which the authors show that the rescaled energy $(1-s)[u]^p_{s,p}$ converges pointwise to the Dirichlet $p$-energy as $s \to 1^-$. In the same spirit, Maz'ya and Shaposhnikova \cite{Mazya02} also proved the pointwise convergence of the rescaled \((s,p)\)-Gagliardo seminorm \(s[u]^p_{s,p}\) to the \(L^p\)-norm as \(s\to 0^+\). 

This has motivated a vast body of work \cite{Ambrosio20,Brasco15,Crismale23,Goldman25,DeLuca24,Ponce04} to understand this convergence in the robust variational framework of $\Gamma$-convergence. In \cite{Crismale23}, the analysis for the quadratic case ($p=2$) was performed on a bounded domain $\Omega \subset \mathbb{R}^d$ with homogeneous Dirichlet boundary conditions. The authors proved that as $s \to 1^-$, the rescaled energy $(1-s)[u]^2_{s,2}$ $\Gamma$-converges (with respect to the strong \(L^2\) topology) to the classical Dirichlet integral. In the other direction, as $s \to 0^+$, they established that $s[u]^2_{s,2}$ $\Gamma$-converges (with respect to the weak \(L^2\) topology) to the $L^2$-norm. Since then, these results have been extended in various directions, including general exponents $p\in[1,+\infty)$ and energies with non-trivial weights \cite{Pagliari24, Kubin25}. 

The first goal of this paper is to extend the variational analysis of $(s,p)$-Gagliardo seminorms to the periodic framework. We consider the space $L^p(\T)$ of $T$-periodic functions on $\mathbb{R}^d$, for \(p\in[1,+\infty)\), and define our energy functional as the \emph{density} of the $(s,p)$-Gagliardo seminorm over the periodic cell $Q_T\coloneqq [0,T]^d$, denoted by 
\begin{equation} \label{eq: seminorma sp}
	\F^s_p(u)\coloneqq \int_{Q_T}\int_{\R^d} \frac{|u(x)-u(y)|^p}{|x-y|^{d+sp}} \,\ud y \ud x.
\end{equation}
Our main objective is to rigorously characterise the $\Gamma$-limits of this functional as $s \to 0^+$ and $s \to 1^-$ for any exponent $p \in [1, +\infty)$.

Our first main result is \Cref{thm: G-conv per s a 0}, which addresses the limit of the functional in \cref{eq: seminorma sp} as $s \to 0^+$. We prove that the rescaled energy $s\F_p^s$ $\Gamma$-converges in the weak topology of $L^p(\T)$ (up to a multiplicative constant) to the functional 
\[
\F^0_p(u)\coloneqq \iint_{Q_T^2} |u(x)-u(y)|^p\, \ud x \ud y.
\] 
Although this functional differs from the one obtained in the bounded domain case, it shares its core properties of non-locality and the loss of information about the gradient. Our proof strategy relies on establishing sharp energy estimates (\Cref{prop: energy estimates}) that relate the energy to its \emph{core} and \emph{tail} components. These estimates are crucial for proving the compactness of minimising sequences and for establishing the lower bound in \Cref{thm: G-conv per s a 0}, which is carried out using a combination of Fatou Lemma and the weak lower semicontinuity of the $L^p$-norm. The upper bound then follows from pointwise convergence together with a density argument.

Our second contribution (\Cref{thm: G-conv per s a 1}) concerns the limit as $s \to 1^-$. We establish that the rescaled energy $(1-s)\F_p^s$ $\Gamma$-converges in the strong topology of \(L^p(\T)\) to a multiple of the Dirichlet $p$-energy $\int_{Q_T} |\nabla u|^p \, \ud x$ for \(p\in(1,+\infty)\), and to \(|\mathrm{D}u|(Q_T)\) for \(p=1\). This result agrees with those in \cite{Brezis01,Brezis02,Kubin25} and shows that the local behaviour of the limit functional is again preserved in the periodic setting. The proof requires a delicate analysis. We first establish compactness by adapting the techniques of \cite{Crismale23, Kubin25} to the periodic case. This relies on a key estimate (\Cref{lemma: stime kubin 2}) that controls the $L^p$-modulus of continuity, which yields strong compactness via the Fréchet-Kolmogorov Theorem. The lower bound is then proved for mollified functions using a Taylor expansion argument, while the upper bound follows from pointwise convergence together with a density argument.

In the last Section, we analyse the one-dimensional minimisers of the energies $\F^s_p$ and their limit $\F^0_p$. Specifically, our problem consists of minimising the one-dimensional energy $\mathcal F^s_p$ among $T$-periodic piecewise affine functions $u$ whose distributional derivative consists of a constant absolutely continuous part and a singular part with opposite sign and quantised jumps; we will show that the unique minimiser (up to translations) of this problem is the function with equispaced jumps. 

Minimising \(s\)-fractional seminorms with constraints on the slopes of the functions has been the subject of extensive analysis. In the context of phase separation models, such analysis has been carried out in \cite{Giuliani11}, where admissible functions are assumed to have only the slopes \(+1\) and \(-1\). This symmetry assumption on the slopes allows one to prove the minimality of the periodic configuration through the so-called {\it reflection positivity} technique, a tool borrowed from field theory, and already exploited in \cite{Giuliani07,Giuliani08} to prove the periodicity of minimisers in different but related non-linear models. Our case corresponds to having slopes $1$ and $-\infty$, and it is non-trivial to extend the reflection positivity technique from the case of symmetric slopes to that of non-symmetric ones. 

Therefore, we adopt a completely different approach, more inspired by \cite{Goldman25,DeLuca24}. Before describing it, we briefly discuss the physical motivation of our analysis, namely the behaviour of the elastic energy induced by edge dislocations at semi-coherent interfaces \cite{Fanzon20}. This subject was originally addressed by van der Merwe \cite{van50}, where, in the framework of 2D linear isotropic elasticity, the author computes the sharp energy density {\it assuming} a periodic distribution of dislocations at the interface \cite{Read50}. In this paper, we aim at {\it proving} the minimality of the periodic configuration, working within the Nabarro-Peierls formalism \cite{Garroni09,Garroni05}. More precisely, minimising the 2D bulk elastic energy subject to a mismatch at the interface between two crystal lattices reduces to a 1D trace energy depending solely on the boundary datum. Since the bulk elastic energy is quadratic, this effective boundary energy is precisely the fractional seminorm \(\mathcal{F}^\frac 1 2_2\) \cite{cinesi20}. Since such a fractional value (\(s\!=\!\frac{1}{2}\)) is in some sense critical, as we will comment later, the energy is \emph{a priori} infinite for discontinuous functions. For this reason, at the beginning of \Cref{sez: minimi rigidi}, we will present two possible solutions: on one hand one can consider the natural ambient space of \(H^\frac{1}{2}\)-functions, on the other hand one can apply a standard regularisation procedure to these functions. We will briefly present the first approach applied to the linear elasticity model. Nevertheless, we choose to adopt the second approach in view of the following consideration. Indeed, concerning the admissible displacement traces, we adopt a strong modelling simplification: we assume the function has a constant slope of \(1\) in the elastic phase, and \(-\infty\) at any dislocation site. In other words, dislocations are represented by negative Dirac deltas, whose quantised nature reflecting the discrete structure of the defects. While the minimality of the periodic distribution under these exact constraints was proven in \cite{Goldman25} for \(p=2\) and \(s\in(0,\frac{1}{2})\), in this paper we generalise the result to the non-linear regime \(p\neq 2\).

The focus on a one-dimensional scalar fractional energy is rigorously motivated by a dimensional reduction strategy detailed at the beginning of \Cref{sez: minimi rigidi}. We first consider the classical physical setting of linear isotropic elasticity (which corresponds to \(p=2\) and \(s=\frac{1}{2}\)). By modelling edge dislocations as straight, parallel lines in a 3D crystal, translation invariance allows us to reduce the bulk elastic energy to a 2D cross-section. Under standard mirror-symmetry assumptions across the slip plane, the vertical stress components vanish at the interface. Consequently, via the Dirichlet-to-Neumann map, the minimised 2D bulk elastic energy reduces to the one-dimensional \(H^{\frac{1}{2}}\)-Gagliardo seminorm of the scalar slip displacement along the interface. 

Motivated by this, we generalise the framework to a $d$-dimensional domain by considering the bulk Dirichlet \(p\)-energy of the displacement. The presence of defects modelled as straight, parallel \((d-2)\)-dimensional manifolds allows us to exploit translation invariance, naturally reducing the analysis to a 2D cross-section orthogonal to the defects. On this 2D plane, the defects impose a measure-valued curl constraint confined to the direction of a fixed vector (that we take to be $e_1$). By the strict convexity of the \(p\)-energy, we can assume the last \(d-1\) components of the displacement to be zero. This seamlessly reduces the $d$-dimensional vectorial problem to a purely 2D scalar one.

However, when deriving the effective 1D trace energy of this \(p\)-harmonic extension via integration by parts, the exact reduction to the \((1-\frac{1}{p}, p)\)-Gagliardo seminorm holds true only for \(p=2\). For \(p \neq 2\), the trace operator is no longer an isometry (up to a constant) on \(p\)-harmonic extensions, and the 1D fractional seminorm is only \emph{equivalent} to the minimised 2D bulk energy. Nevertheless, the energy given by the 1D \((s,p)\)-Gagliardo seminorm is still mathematically intriguing, therefore we select it as the primary object of our mathematical investigation.  

We first focus on the {\it sub-critical case} $sp<1$. To this end, since the energy is invariant under translations, by a slight abuse of notation we can express it directly as a function of the position of the jump points, which are not necessarily distinct. More precisely, a \emph{configuration} is represented by a vector ${\bf X}=(x_1,\ldots,x_T)$ with $x_l\in[0,T)$. To any configuration \(\bold X \), we can associate a unique function $u=u[\bold{X}]$ (up to horizontal and vertical translations) satisfying $\mathrm{D}u=\mathcal{L}^1-\sum_{k\in\Z}\sum_{l=1}^{T} \delta_{x_l+kT}$; we define $\F^s_p[\bold{X}]:=\F^s_p(u[\bold{X}])$. In order to prove that the only minimiser of $\mathcal{F}^s_p$ is the equispaced configuration \(\bold{X}^\mathrm{eq}\), we follow the scheme in \cite{Goldman25}. First, we compute the first variation of $\mathcal{F}^s_p$ with respect to perturbations of the positions of the jump points (which we refer to as \emph{rigid variations}) and we show that it vanishes for the equispaced configuration, thus showing it is a critical point of $\mathcal{F}^s_p$. Then, we prove that configurations of overlapping jumps cannot be critical points of the energy. This is achieved with \Cref{lemma: cuspidi} for \(p>1\), showing that the first variation has a cusp at overlapping jumps, and with \Cref{lemma: cuspidi p=1} for \(p=1\), where we show that separating overlapping jumps by a sufficiently small distance lowers the energy. 
Finally, we compute the second variation of $\mathcal{F}^s_p$ on configurations with distinct jump points, proving that it is positive-definite. This immediately implies that the equispaced configuration is indeed the unique minimiser of the energy $\mathcal{F}^s_p$. It is worth noting that for $p=2$, the first variation of $\mathcal{F}^s_p$ with respect to rigid variations coincides (at regular points of $u$) with the $s$-fractional Laplacian of the function $u$. In our non-linear case this is no longer true; indeed, it is easy to prove that the first variation of the \((s,p)\)-Gagliardo seminorm is the classical \(s\)-fractional \(p\)-Laplace operator for \emph{smooth variations}, but this is not the case for rigid variations when \(p\neq2\). The first variation of the functional \(\F^s_p\) with respect to perturbations of the jump points is 
\[
\left(- \D \right)^s_p u(x) \coloneqq 2 \int_\R \frac{\left| u(x)-u(y)+1\right|^p-\left|u(x)-u(y)\right|^p}{\left|x-y\right|^{1+sp}} \,\ud y,
\]
defined for the functions we described above. We remark that, in this case, perturbations of the jump points correspond to inner variations of the function $u$, as shown in \cref{def: var rigida}.

We then consider the \emph{critical} and \emph{super-critical cases} $sp\ge 1$. Formally, one may adopt the same strategy used in the sub-critical case to deduce that the unique minimiser (up to translations) of $\F^{s}_p$ is the equispaced configuration. The main issue in this setting is that the energy $\F^{s}_p$ is identically $+\infty$ for every configuration (indeed \( u[ \bold{X}] \notin W^{s,p}(\mathbb{T}^1_T) \) for any \(sp\geq 1\)). To overcome this issue, we adopt a standard regularisation approach. We analyse the energy of a regularised functional \(\F^{s,\eps}_p[\bold{X}]\coloneqq \F^s_p(u[\bold{X}]\ast\rho_\eps)\), where $\rho_\eps$ is a standard mollifier, under the further assumption that any two points in \(\bold{X}\) are separated by at least \(4\eps\). Then, as in the sub-critical case, we compute the first and the second variations of the functional \(\mathcal F_p^{s,\eps}\), with respect to perturbations of jump locations, which are now the centres of the regularised transitions. Contrary to the sub-critical case, as shown in \cref{rmk: variazione rigida moscia}, such perturbations correspond to external variations of the function $u^\eps$, so that the first variation of $\mathcal F_p^{s,\eps}$ is exactly the standard $s$-fractional $p$-Laplacian \cite{Mazon16}
\begin{equation*}
	\left(- \Delta \right)^s_p u^\eps(x) \coloneqq p \int_{\R} \frac{\left\lvert u^\eps(x)-u^\eps(y)\right\rvert^{p-2}(u^\eps(x)-u^\eps(y))}{|x-y|^{1+sp}} \,\ud y.
\end{equation*}
We are then able to prove that the equispaced configuration is the unique minimiser for \(\F^{s,\eps}_p\) in the class of configurations with minimal distance \(4\eps\). The main issue in generalising this result to the whole space of configurations is that the sign of the second variation of \(\F^{s,\eps}_p\) is not clear when jump points are close to each other. Nevertheless, we can recover the full result in the \emph{critical regime} \(sp=1\). In particular, we prove in \Cref{prop: caso critico salvo} that the energy grows logarithmically with respect to the distance between two jump points as they approach each other; therefore, it is always energetically favourable to maintain a set positive minimum distance.  

Finally, we study the minimality of the equispaced configuration for the limit functional $\mathcal F^0_p$ as well. To this end, we use a different approach inspired by M{\"u}ller \cite{Müller93}, in which the existence of periodic minimisers was established using convexity techniques for a local energy like the \(L^2\)-norm, on bounded domains, in the case of slopes having equal moduli and opposite signs, which was later extended to general opposite slopes in Ren and Wei \cite{RenWei03}. Thanks to the \(\Gamma\)-convergence result, we already know that the equispaced configuration is a minimiser of the energy \(\F^0_p\); however, the convexity of the functional is the key property ensuring that \(\bold{X}^{\mathrm{eq}}\) is the \emph{unique} minimiser.

The paper is structured as follows. In \Cref{sez: notazione}, we introduce some notation regarding periodic fractional spaces and some useful energy estimates. In \Cref{sez: gamma convergenza}, we study the \(\Gamma\)-convergence of our functional as \(s\to0^+\) and \(s\to 1^-\). In \Cref{sez: minimi rigidi}, we characterise the periodic minimisers for \(\F^s_p\) in the sub-critical (\(sp <1\)) and in the critical and super-critical (\(sp\geq 1\)) cases, and for its \(\Gamma\)-limit \(\F^0_p\).

\section{Notation and preliminary results}
\label{sez: notazione} 

In this paper, we consider the energy given by the \((s,p)\)-Gagliardo fractional seminorm for \(T\)-periodic functions. 
Let \(p\in[1,+\infty)\) and \(T>0\) be the periodicity of the functions we will consider and let us denote by \((e_1,\ldots,e_d)\) the canonical basis of \(\R^d\). We denote by \(\T\) the \(d\)-dimensional \(T\)-torus and we define the space of \(L^p\) functions on \(\T\) as
\begin{equation*}
	L^p(\T)\coloneqq \left\{u:\R^d\to\R \; \text{measurable} \;\left|\; 
	\begin{aligned}
		&u(x+T e_i)=u(x) \text{ for a.e. }x\in\R^d \\&\forall i=1,\ldots,d \text{ and } u|_{Q_T}\in L^p(Q_T) 
	\end{aligned}
	\right.\right\},
\end{equation*}
endowed with the norm 
\[
\|u\|_{L^p(\T)}\coloneqq\bigg(\int_{Q_T} |u(x)|^p \ud x\bigg)^{\frac{1}{p}}.
\]
In addition, we denote by \(\fint u\) the average of \(u\) on the periodic cell \(Q_T\). To be more precise, the average of \(u\) is the quantity
\[
\fint_{Q_T} u(x)\, \ud x\coloneqq\frac{1}{T^d}\int_{Q_T} u(x)\, \ud x.
\]
Let \( s \in (0,1)\). We similarly define the \(T\)-periodic fractional Sobolev space
\begin{equation*}
	W^{s,p}(\T)\coloneqq \left\{ u\in L^p(\T) \,|\, [u]_{s,p} <+\infty \right\},
\end{equation*}
where we denote by \([u]_{s,p}\) the \((s,p)\)-periodic Gagliardo seminorm:
\begin{equation*}
	[u]_{s,p} \coloneqq \left(\int_{Q_T}\int_{\R^d} \frac{|u(x)-u(y)|^p}{|x-y|^{d+sp}} \, \ud y \ud x\right)^{\frac{1}{p}}.
\end{equation*}
On these spaces, we use the norm 
\[
\|u\|_{W^{s,p}(\T)}\coloneqq \|u\|_{L^p(\T)}+ [u]_{s,p}.
\]
\begin{remark} 
    For functions \(u\in L^p(\Omega)\), where \(\Omega\) is an open set of \(\R^d\) the usual definition of the Gagliardo seminorm is the following:
    \[
    \left(\int_{\R^d}\int_{\R^d} \frac{|u(x)-u(y)|^p}{|x-y|^{d+sp}} \, \ud y \ud x\right)^{\frac{1}{p}}.
    \]
    Since the \(L^p(\T)\) functions we consider are defined on \(\R^d\) but they do not belong to \(L^p(\R^d)\), this definition of seminorm is not well-posed. Indeed
    \[
    \left(\int_{\R^d}\int_{\R^d} \frac{|u(x)-u(y)|^p}{|x-y|^{d+sp}} \, \ud y \ud x\right)^{\frac{1}{p}}=+\infty, \text{ for every non-constant } u\in L^p(\T),
    \]
    as shown in \cite{Brezis02_2}. To avoid this complication, we call \emph{Gagliardo seminorm} what in reality is the density of the Gagliardo seminorm on the periodic cell \(Q_T\).
\end{remark}

\begin{remark} \label{rmk: inclusioni sobolev}
    The spaces \(L^p(\T)\) and \(W^{s,p}(\T)\) inherit the properties of the usual \(L^p\) and \(W^{s,p}\) spaces (see \cite{Palatucci12,Leoni23, Triebel78} for references about fractional Sobolev spaces). They are Banach spaces, reflexive for \(p\in(1,+\infty)\) and the usual \(L^p\) and Sobolev embeddings hold. Indeed, thanks to periodicity, we only have to consider the restriction of \(u\) to \(Q_T\); therefore, these periodic spaces inherit the same inclusions that hold for \(L^p\) and Sobolev spaces on bounded domains:
    \[
    L^q(\T) \hookrightarrow L^p(\T), \qquad  \text{ where }  1\leq p\leq q\leq +\infty
    \]
    and for \(p\in[1,+\infty)\) 
    \[
    W^{s_2,p}(\T) \hookrightarrow W^{s_1,p}(\T), \qquad  \text{ where } 0\leq s_1\leq s_2 \leq 1.
    \]
\end{remark}
We also recall the space of periodic functions of \emph{bounded variation}: 
 \begin{equation*}
	 \mathrm{BV}(\T)\coloneqq \left\{ u:\R^d\to\R \text{ measurable}\;\left|\;
	\begin{aligned}
		&u(x+T e_i)=u(x) \text{ for a.e. }x\in\R^d \\&\forall i=1,\ldots,d \text{ and } u|_{Q_T}\in \mathrm{BV}(Q_T) 
	\end{aligned}
	\right.\right\},
\end{equation*}
endowed with the norm 
\[
\|u\|_{\mathrm{BV}(\T)}= \|u\|_{L^1(\T)}+|\mathrm{D}u|(\T),
\]
where \(|\mathrm{D}u|(Q_T)=\int_{Q_T}\ud |\mathrm{D}u|\). Finally we denote by
 \begin{equation*}
	 C^\infty(\T)\coloneqq \left\{ u \in C^\infty (\R^d)\;\left|\; u(x+T e_i)=u(x) \,\forall x\in\R^d, \forall i=1,\ldots,d \right.\right\}
\end{equation*}
the space of \(T\)-periodic smooth functions.

In this paper, we are going to discuss the \(\Gamma\)-convergence of the energy functional
\begin{equation*}
	\F_p^s(u) = [u]_{s,p}^p
\end{equation*}
as \(s \to 0^+\) and \( s \to 1^-\). Trivially, the functional is finite whenever \(u\in W^{s,p}(\T)\). To achieve the \(\Gamma\)-convergence result we will rescale the functional appropriately (as discussed in \cite{Brezis01, Brezis02, Crismale23}); in particular, we scale the functional by \(s\) in the limit as \(s \to 0^+\) and by \((1-s)\) as \(s \to 1^-\).

We now show some preliminary results and estimates that will be useful in the following pointwise and \(\Gamma\)-convergence results. Given a set \(E \subset \R^d\), let us denote by \([u]_{s,p}|_E\) the \emph{cut-off} Gagliardo seminorm, i.e.
\[
    [u]_{s,p}|_E\coloneqq \int_{Q_T}\int_E \frac{|u(x)-u(y)|^p}{|x-y|^{d+sp}}\, \ud y \ud x.
\]
Moreover, we set for brevity \(B_R \coloneqq B_R(0)\) for every \(R>0\) and, from now on, we will denote \(\omega_d\coloneqq |\mathbb{S}^{d-1}|\).
\begin{proposition}[Energy estimates] \label{prop: energy estimates}
	Let \(p\in[1,+\infty)\) and \(u\in L^p(\T)\). Then for all \(s\in(0,1)\) and for all \(R>2T\sqrt{d}\) the following estimates hold:
	\[
[u]_{s,p}^p|_{B_R}+ \frac{d\omega_d \left(\frac{R}{T}+2\sqrt{d}\right)^{-sp}}{sp(Tc_1(R,d))^{d+sp}}\F_p^0(u) \leq \F_p^s (u) \leq [u]_{s,p}^p|_{B_R} + \frac{d\omega_d \left(\frac{R}{T}-2\sqrt{d}\right)^{-sp}}{sp(Tc_2(R,d))^{d+sp}}\F_p^0(u) ,
	\]
    with constants 
    \[
    c_1(R,d)\coloneqq\frac{R+2T\sqrt{d}}{R+T\sqrt{d}} \qquad \text{and} \qquad c_2(R,d)\coloneqq\frac{R-2T\sqrt{d}}{R-T\sqrt{d}}.
    \]
\end{proposition}
\begin{proof}
	We define the partition of \(\R^d\) given by cubes as
	\[
	\R^d = \bigcup_{I\in\Z^d} Q_I,
	\]
	where each \(Q_I\coloneqq (TI+Q_T)\). We can write the energy as the sum over the cubes \(Q_I\)
	\begin{align*}
		\F_p^s(u) =& \int_{Q_T}\int_{\R^d} \frac{|u(x)-u(y)|^p}{|x-y|^{d+sp}} \, \ud y \ud x \\
		=& \sum_{I\in\Z^d}\int_{Q_T}\int_{Q_I} \frac{|u(x)-u(y)|^p}{|x-y|^{d+sp}} \, \ud y \ud x \\
		=& \sum_{I\in\Z^d}\int_{Q_T}\int_{Q_T} \frac{|u(x)-u(y)|^p}{|x-y-TI|^{d+sp}} \, \ud y \ud x,
	\end{align*}
	where in the last step we apply the change of variables given by \(y\mapsto y+TI\) and use the periodicity of \(u\). We write the energy as the sum
	\[
	\F_p^s(u) =  \sum_{I\in\Z^d} f_I(u),
	\]
	with
	\[
	f_I(u) \coloneqq \int_{Q_T}\int_{Q_T} \frac{|u(x)-u(y)|^p}{|x-y-TI|^{d+sp}} \, \ud y \ud x.
	\]
	Let \(R>2T\sqrt{d}\), we observe that for \(W^{s,p}(\T)\) functions \([u]^p_{s,p}|_{B_R}\) is finite by definition. The other component of the energy, which we refer to as the \emph{tail}, is controlled from above and below by 
	\begin{equation} \label{eq: stima code dall'alto e dal basso}
	 \sum_{\substack{I\in\Z^d \\ |I|>\frac{R}{T}+\sqrt{d}}} f_I(u)\leq [u]^p_{s,p}|_{B_R^c}\leq \sum_{\substack{I\in\Z^d \\ |I|>\frac{R}{T}-\sqrt{d}}} f_I(u).
    \end{equation}
\begin{figure}
	\centering 
	\makebox[\textwidth]{
	\scalebox{.8}{
	\begin{tikzpicture}
        \draw [fill=black!5](270:6.1)--(0,0)--(180:6.1) arc (180:270:6.1);
		\draw [fill=black!12](270:4.7)--(0,0)--(180:4.7) arc (180:270:4.7);
		\draw [fill=black!20](270:3.3)--(0,0)--(180:3.3) arc (180:270:3.3);
		\node at (-1,0) [above=1]{\(T\)};
		\node at (-6.1,0) [above]{\(R\!+\!T\sqrt{d}\)};
        \node at (-3.3,0) [above]{\(R\!-\!T\sqrt{d}\)};
		\node at (-4.7,0) [above=1] {\(R\)};
		
		\draw[step=1, black] (0,0) grid (-7,-7);
		
		\draw[thick,<-] (1,0)--(-8,0);
		\draw[thick,<-] (0,1)--(0,-8);
		
		\draw[very thick] (0,-3) rectangle (-1,-4) ;
		\draw[arrows={-Stealth}] (0,0) -- (-1,-4) node[above=-2, pos=.6, sloped]{\(TI_1\)};

		\draw[very thick] (-4,-1) rectangle (-5,-2) ;
		\draw[arrows={-Stealth}] (0,0) -- (-5,-2) node[above=-2, pos=.65, sloped]{\(TI_2\)};

        \draw[very thick] (-4,-4) rectangle (-5,-5) ;
		\draw[arrows={-Stealth}] (0,0) -- (-5,-5) node[above=-2, pos=.71, sloped]{\(TI_3\)};
	\end{tikzpicture}
}
\hspace{0cm}
\scalebox{.8}{
		\begin{tikzpicture} 
        \draw [fill=black!5](90:6.1)--(0,0)--(0:6.1) arc (0:90:6.1);
		\draw [fill=black!12](90:4.7)--(0,0)--(0:4.7) arc (0:90:4.7);
		\draw [fill=black!20](90:3.3)--(0,0)--(0:3.3) arc (0:90:3.3);
		\node at (1,0) [below=2.5]{\(T\)};
		\node at (4.7,0) [below=2.5]{\(R\)};
		\node at (3.3,0) [below] {\(R\!-\!T\sqrt{d}\)};
        \node at (6.1,0) [below] {\(R\!+\!T\sqrt{d}\)};
		
		\draw[step=1, black] (0,0) grid (7,7);
		
		\draw[thick,->] (-1,0)--(8,0);
		\draw[thick,->] (0,-1)--(0,8);
		
		\draw[very thick] (0,0) node[right=.5cm, above=.2cm]{\(Q_T\)} rectangle (1,1) ;
		
		\draw[very thick] (1,4) rectangle (2,5) ;
		\draw[arrows={-Stealth}] (0,0) -- (1,4) node[above=-2, pos=.8, sloped]{\(TI_1\)};

		\draw[very thick] (5,2) rectangle (6,3) ;
		\draw[arrows={-Stealth}] (0,0) -- (5,2) node[above=-2, pos=.71, sloped]{\(TI_2\)};

        \draw[very thick] (5,5) rectangle (6,6) ;
		\draw[arrows={-Stealth}] (0,0) -- (5,5) node[above=-2, pos=.8, sloped]{\(TI_3\)};
	\end{tikzpicture}
}
}
    \caption{The key observation is that a vector \(I \in \Z^d\) always points to the corner with smallest components of the hypercube \(Q_I\). This means that a hypercube in a sector with any negative coordinate is \virgolette{closer} to the origin than a hypercube in the sector with all positive coordinates having a vector of the same norm. We need to consider this fact to not over/under-estimate \([u]_{s,p}|_{B_R^c}\). In the figure above the vectors on the left are the reflections, centred at the origin, of those on the right.}
\label{fig: stima sui quadratini 1}
\end{figure}
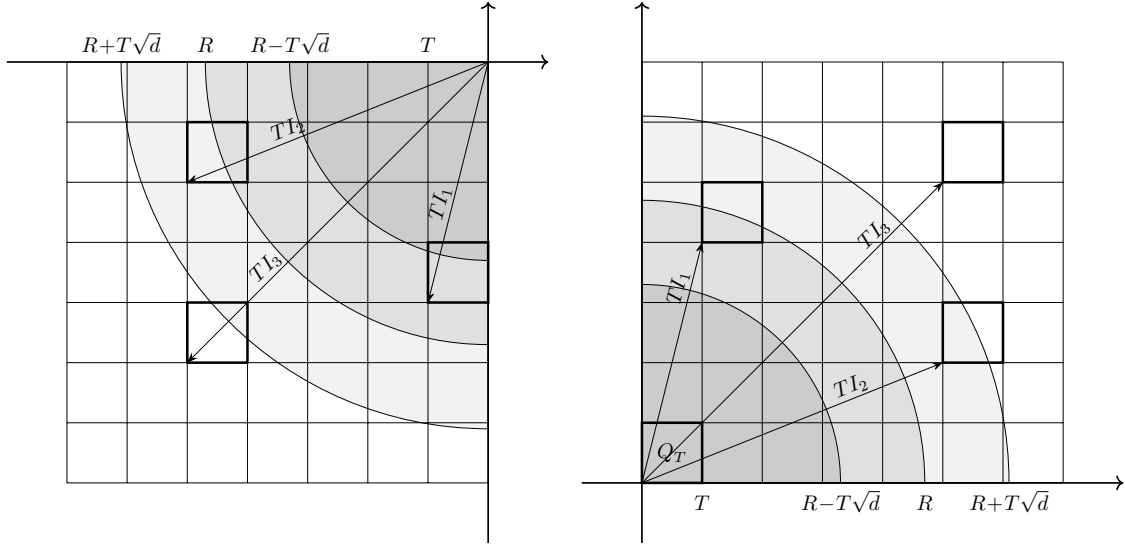 %
Also note that we can write the following estimates:
\begin{equation} \label{eq: stime su |x-y-TI|}
	(|I|-\sqrt{d})T \leq |x-y- TI| \leq (|I|+\sqrt{d})T, \qquad \forall I\in\Z^d; \, x,y\in Q_T,
\end{equation}
which are non-trivial for \(|I|> \sqrt{d}\).
As a consequence, we get
\begin{align} \label{eq: stime somme |I| dal basso}
	\sum_{\substack{I\in\Z^d \\ |I|>\frac{R}{T}+\sqrt{d}}} f_I(u) & \geq \sum_{\substack{I\in\Z^d \\ \nonumber |I|>\frac{R}{T}+\sqrt{d}}}\int_{Q_T}\int_{Q_T} \frac{|u(x)-u(y)|^p}{(|I|+\sqrt{d})^{d+sp}T^{d+sp}} \, \ud y \ud x \\\nonumber
	&= \frac{\F_p^0(u)}{T^{d+sp}} \sum_{\substack{I\in\Z^d \\ |I|>\frac{R}{T}+\sqrt{d}}} \frac{1}{(|I|+\sqrt{d})^{d+sp}} \\ \nonumber
	& \geq \frac{\F_p^0(u)}{T^{d+sp}} \sum_{\substack{I\in\Z^d \\ |I|>\frac{R}{T}+\sqrt{d}}} \frac{1}{|I|^{d+sp}\left(1+\frac{\sqrt{d}}{\frac{R}{T}+\sqrt{d}}\right)^{d+sp}} \\
	&= \frac{\F_p^0(u)}{(Tc_1(R,d))^{d+sp}} \sum_{\substack{I\in\Z^d \\ |I|>\frac{R}{T}+\sqrt{d}}} \frac{1}{|I|^{d+sp}},
\end{align}
where in the first inequality we used the upper bound in \cref{eq: stime su |x-y-TI|}, and we have set  
\[ 
c_1(R,d)\coloneqq\frac{R+2T\sqrt{d}}{R+T\sqrt{d}}.
\]
Analogously, by the lower bound in \cref{eq: stime su |x-y-TI|}, we have
\begin{align} \label{eq: stime somme |I| dall'alto}
	\sum_{\substack{I\in\Z^d \\ |I|>\frac{R}{T}-\sqrt{d}}} f_I(u) &\leq  \sum_{\substack{I\in\Z^d \\ |I|>\frac{R}{T}-\sqrt{d}}} \int_{Q_T}\int_{Q_T} \nonumber \frac{|u(x)-u(y)|^p}{(|I|-\sqrt{d})^{d+sp}T^{d+sp}} \, \ud y \ud x \\ \nonumber
	&= \frac{\F_p^0(u)}{T^{d+sp}} \sum_{\substack{I\in\Z^d \\ |I|>\frac{R}{T}-\sqrt{d}}} \frac{1}{(|I|-\sqrt{d})^{d+sp}} \\ \nonumber
	& \leq \frac{\F_p^0(u)}{T^{d+sp}} \sum_{\substack{I\in\Z^d \\  |I|>\frac{R}{T}-\sqrt{d}}} \frac{1}{|I|^{d+sp}\left(1-\frac{\sqrt{d}}{\frac{R}{T}-\sqrt{d}}\right)^{d+sp}} \\ 
	&= \frac{\F_p^0(u)}{(Tc_2(R,d))^{d+sp}} \sum_{\substack{I\in\Z^d \\ |I|>\frac{R}{T}-\sqrt{d}}} \frac{1}{|I|^{d+sp}},
\end{align}
	where we have set 
    \[
    c_2(R,d)\coloneqq\frac{R-2T\sqrt{d}}{R-T\sqrt{d}}.
    \]
	We now observe that we can control from below the sum in \cref{eq: stime somme |I| dal basso} with 
	\begin{equation*}
		\sum_{\substack{I\in\Z^d \\ |I|>\frac{R}{T}+\sqrt{d}}} \frac{1}{|I|^{d+sp}} \geq \int_{B_{\frac{R}{T}+2\sqrt{d}}^c} \frac{1}{|x|^{d+sp}}\,\ud x = d\omega_d \int_{\frac{R}{T}+2\sqrt{d}}^{+\infty} \frac{1}{\rho^{1+sp}}\, \ud \rho = \frac{d\omega_d}{sp\left(\frac{R}{T}+2\sqrt{d}\right)^{sp}}
	\end{equation*}
	and from above the one in \cref{eq: stime somme |I| dall'alto} with 
	\begin{equation*}
		\sum_{\substack{I\in\Z^d \\ |I|>\frac{R}{T}-\sqrt{d}}} \frac{1}{|I|^{d+sp}} \leq \int_{B_{\frac{R}{T}-2\sqrt{d}}^c} \frac{1}{|x|^{d+sp}}\,\ud x = d\omega_d \int_{\frac{R}{T}-2\sqrt{d}}^{+\infty} \frac{1}{\rho^{1+sp}}\, \ud \rho = \frac{d\omega_d}{sp\left(\frac{R}{T}-2\sqrt{d}\right)^{sp}}.
	\end{equation*}
	By \cref{eq: stima code dall'alto e dal basso,eq: stime somme |I| dal basso,eq: stime somme |I| dall'alto}, we obtain the claim.

\end{proof}

\section{The \(\Gamma\)-convergence results}
\label{sez: gamma convergenza}

\subsection{Pointwise convergence and \(\Gamma\)-convergence for \(s\to0^+\)}
\begin{proposition}[Pointwise convergence for \(s\to0^+\)] \label{prop: puntuale a 0}
    Let \(p\geq 1\), \((s_n)_{n \in \N} \subset (0,1)\), such that \(s_n \to 0^+\) as \(n \to +\infty\). Let \(u\in W^{s_1,p}(\T)\), then it holds true
    \[
    \lim_{\substack{n \to +\infty}} s_n \F^{s_n}_p(u) = \frac{d \omega_d}{pT^d}\F^0_p(u).
    \]
\end{proposition}
\begin{proof}
Firstly, we observe that the limit is well defined in view of \cref{rmk: inclusioni sobolev}. By the upper bound in \Cref{prop: energy estimates}, for every \(R>0\) we have
\begin{align*}
\limsup_{\substack{n \to +\infty}} s_n \F^{s_n}_p(u) &\leq \limsup_{\substack{n \to +\infty}} \left( s_n[u]_{s_n,p}^p|_{B_R} + \frac{d\omega_d \left(\frac{R}
{T}-2\sqrt{d}\right)^{-s_n p}}{p(Tc_2(R,d))^{d+s_n p}}\F_p^0(u)\right)\\
&=\frac{d\omega_d}{p(Tc_2(R,d))^d} \F^0_p(u).
\end{align*}
Analogously, by using the lower bound in \Cref{prop: energy estimates}, for every \(R>0\) we have 
\begin{align*}
\liminf_{\substack{n \to +\infty}} s_n \F^{s_n}_p(u) &\geq \liminf_{\substack{n \to +\infty}} \left( s_n[u]_{s_n,p}^p|_{B_R} + \frac{d\omega_d \left(\frac{R}{T}+2\sqrt{d}\right)^{-s_n p}}{p(Tc_1(R,d))^{d+s_np}}\F_p^0(u) \right)\\
&=\frac{d\omega_d}{p(Tc_1(R,d))^d} \F^0_p(u).
\end{align*}
Combining the two inequalities, we get the claim sending \(R \to +\infty\).

\end{proof}
In the next theorem we show that \(\F^0_p\) is actually the \(\Gamma\)-limit of \(s\F^s_p\).

\begin{theorem}[\(\Gamma\)-convergence for \(s\to 0^+\)] \label{thm: G-conv per s a 0}
	Let \(p\in[1,+\infty)\) and \((s_n)_n \subset (0,1)\) such that \(s_n \to 0^+\) as \(n \rightarrow +\infty\).

		\emph{1. Compactness.} Assume \(p\neq1\) and let \((u_n)_n\subset L^p(\T)\) be such that there exists a constant \(C \in \R\) such that
        \begin{equation} \label{eq: Hyp for compactness}
            \sup_{\substack{n\in\N}} s_n \F_p^{s_n} \left({\textstyle u_n}\right) \leq C,
        \end{equation}
		then, without relabelling subsequences, \(u_n-\fint u_n \rightharpoonup u\) in \(L^p(\T)\) for some \(u \in L^p(\T)\), and \(\fint u =0\).
		
		\emph{2. \(\Gamma\)-\(\liminf\) inequality.} For all \(u \in  L^p(\T)\) and all \((u_n)_n\subset L^p(\T)\) such that \(u_n \rightharpoonup u\) in \( L^p(\T)\) the following holds:
		\[
		\liminf_{n\rightarrow +\infty} s_n \F_p^{s_n}(u_n) \geq \frac{d\omega_d}{pT^d}\F_p^0(u).
		\]
        
		\emph{3. \(\Gamma\)-\(\limsup\) inequality.} For all \(u \in L^p(\T)\) there exists \((u_n)_n \subset L^p(\T)\) such that \(u_n \rightharpoonup u \) in \( L^p(\T)\) and 
		\[
		\limsup_{n\rightarrow +\infty} s_n \F_p^{s_n}(u_n) \leq \frac{d\omega_d}{pT^d}\F_p^0(u).
		\]
\end{theorem}

\begin{proof}\strut 

	\emph{1. Compactness.} To prove the first point, we can directly apply the lower bound from \Cref{prop: energy estimates}, obtaining that for every \(R>2T\sqrt{d}\)
\begin{align*}
	s_n\F_p^{s_n}(u_n) \geq& s_n[u_n]^p_{s_n,p}|_{B_R}+ \frac{d\omega_d \left(\frac{R}{T}+2\sqrt{d}\right)^{-s_np}\F_p^0(u_n)}{p(Tc_1(R,d))^{d+s_np}}\\
	\geq & \F_p^0(u_n)\inf_{n\in\N} \left(\frac{d\omega_d \left(\frac{R}{T}+2\sqrt{d}\right)^{-s_np}}{p(Tc_1(R,d))^{d+s_np}}\right)\\
	=& C(p,d,R,s_1,T) \F_p^0(u_n).
\end{align*}
	By Jensen inequality
	\begin{align*}
		\F_p^0(u_n) &=\int_{Q_T}\int_{Q_T} |u_n(x)-u_n(y)|^p  \,\ud y \ud x \\
		&\geq T^{d-dp}\int_{Q_T} \left| \int_{Q_T} u_n(x)-u_n(y) \, \ud y \right|^p \ud x \\
		&= T^d\|u_n - {\textstyle\fint} u_n \|_{L^p(\T)}.
	\end{align*} 
	It follows from \cref{eq: Hyp for compactness} that the \(L^p\) norm of the sequence \(\left(u_n-\fint u_n\right)_n\) is bounded uniformly in \(n\). Since \(L^p\), \(p\in(1,+\infty)\), is a reflexive space, weak compactness follows.
	
	\emph{2. \(\Gamma\)-liminf inequality.} Using once again the lower bound in \Cref{prop: energy estimates} we obtain
	\begin{align*} 
		\liminf_{n \rightarrow +\infty} s_n \F_p^{s_n}(u_n) 
        &\geq
		\liminf_{n \rightarrow +\infty} \left( s_n[u_n]^p_{s_n,p}|_{B_R}+ \frac{d\omega_d \left(\frac{R}{T}+2\sqrt{d}\right)^{-s_n p}}{p(Tc_1(R,d))^{d+s_n p}}\F_p^0(u_n)\right) \\
		&\geq \frac{d \omega_d}{p}\liminf_{n \rightarrow +\infty} \F_p^0(u_n) \liminf_{n\to +\infty}\frac{\left(\frac{R}{T}+2\sqrt{d}\right)^{-s_n p}}{T^{d+s_n p}c_1(R,d)^{d+s_n p}}\\
		&=\frac{d \omega_d}{pT^d c_1(R,d)^d} \liminf_{n \rightarrow +\infty} \F_p^0(u_n) \xrightarrow{R\to +\infty} \frac{d \omega_d}{pT^d} \liminf_{n \rightarrow +\infty} \F_p^0(u_n),
	\end{align*}
	where we recall that \(c_1(R,d)\to1\), for \(R\to +\infty\).
    Now, setting \(\tau_h u(x):=u(x+h)\) for \(x,h\in\R^d\), using periodicity, we can write
	\begin{flalign*}
		\liminf_{n \to +\infty} \F_p^0(u_n) &= \liminf_{n \to +\infty} \int_{Q_T}\int_{Q_T} |u_n(x)-u_n(y)|^p \ud y\ud x \\
        &= \liminf_{n \to +\infty} \int_{Q_T}\int_{Q_T} |\tau_h u_n(y)-u_n(y)|^p \ud y\ud h \\
		&\geq \int_{Q_T}  \liminf_{n \to +\infty} \|\tau_h u_n-u_n \|_{L^p(\T)}^p \ud h\\
		&\geq \int_{Q_T} \|\tau_hu-u\|_{L^p(\T)}^p \ud h = \F_p^0(u).
	\end{flalign*}
    In the first inequality we used Fatou's Lemma, in the second one we used the weak lower semicontinuity of the \(L^p\) norm and the weak continuity of the operator \(\tau_h-\mathrm{Id}\). This concludes the proof for the \(\Gamma\)-\(\liminf\) inequality. 
	
	\emph{3. \(\Gamma\)-limsup inequality.} To prove this last result, we start by considering a function \(u \in C^\infty(\T)\), and the \emph{recovery sequence} given by \(u_n \equiv u\). We can now repeat the same computation we used for the pointwise convergence to show the \(\limsup\) inequality for this sequence. To prove the result for \(u\in  L^p(\T)\) we use a density argument as shown in \cite[remark~1.29]{Braides02}.
	
\end{proof}

\subsection{Pointwise convergence and \(\Gamma\)-convergence for \(s\to1^-\)}

\begin{proposition}[Pointwise convergence for \(s\to1^-\)] \label{prop: puntuale a 1}
    Let \(p\geq 1\), \((s_n)_{n \in \N} \subset (0,1)\), such that \(s_n \to 1^-\) as \(n \to +\infty\). Let \(u\in C^\infty(\T)\), then it holds true
    \[
    \lim_{\substack{n \to +\infty}} (1-s_n) \F^{s_n}_p(u) = \F^1_p(u) 
    \coloneqq\frac{2\pi^{\frac{d-1}{2}}\Gamma\left(\frac{p+1}{2}\right)}{p\Gamma\left(\frac{d+p}{2}\right)}\|\nabla u\|_{L^p(\T)}^p \eqqcolon K_{d,p}\|\nabla u\|_{L^p(\T)}^p.
    \] 
\end{proposition}

\begin{proof}
The idea follows from \cite[Theorem~3.6]{Kubin25}, where the same problem is addressed in the non-periodic setting, but with weighted seminorms. Firstly, we fix \(R>2T\sqrt{d}\) and we observe that, thanks to the estimates given by \Cref{prop: energy estimates}, we have 
\[
	(1-s)\F^s_p(u) \leq (1-s)[u]_{s,p}^p|_{B_R} + (1-s)\frac{d\omega_d \left(\frac{R}{T}-2\sqrt{d}\right)^{-sp}}{sp(Tc_2(R,d))^{d+sp}}\F_p^0(u).
\]
First, we observe that in such a case the second term goes to 0 as \(s\to 1^-\) and we can write the integral on the ball \(B_R\) as 
\begin{align*}
		[u]^p_{s,p}|_{B_R} =& \int_{Q_T}\int_{B_1(x)} \frac{|u(x)-u(y)|^p}{|x-y|^{d+sp}} \,\ud y \ud x \\
        &+ \int_{Q_T}\int_{B_R\setminus B_1(x)} \frac{|u(x)-u(y)|^p}{|x-y|^{d+sp}}\, \ud y \ud x\\
        \eqqcolon& I_1+I_2.
\end{align*}
To make sure \(I_2\) is well defined we want \(R\) large enough so that \(B_1(x)\subset B_R\) for all \(x\in Q_T\). To this end the condition \(R>2T\sqrt{d}\) given by \Cref{prop: energy estimates} is sufficient. To evaluate the first integral we write the Taylor expansion of \(u(y)\) around \(x\):
\[
	u(y) = u(x) +\nabla u(x) \cdot (x-y) +  O\left(|x-y|^2\right).
\]
Now, using the following estimate for \(a,b \in \R\) and \(p\geq 1\)
\begin{equation*}
    \left\lvert|a+b|^p-|a|^p\right\rvert\leq C_p (|b||a|^{p-1}+|b|^p),
\end{equation*}
for \(a\coloneqq\nabla u(x) \cdot (x-y)\) and \(b \coloneqq O(|x-y|^2)\), we get
\begin{equation} \label{eq: sviluppo alla p}
		|u(x)-u(y)|^p= |\nabla u(x)\cdot(x-y)|^p + 		O\left(|x-y|^{p+1}\right).
\end{equation}
Replacing this last equation in \(I_1\) we can write (setting \(\nu(x)\coloneqq\nabla u(x)/|\nabla u(x)|\))
\begin{align*}
	I_1 &= \int_{Q_T}\int_{B_1(x)} \frac{\left\lvert \nabla u(x) \cdot (x-y) \right \rvert^p}{|x-y|^{d+sp}}\, \ud y \ud x +C\int_{Q_T}\int_{B_1(x)} |x-y|^{(1-s)p+1-d} \ud y \ud x \\
	&\eqqcolon \int_{Q_T} |\nabla u(x)|^p \int_{B_1} \frac{|\nu(x)\cdot z|}{|z|^{d+sp}} \, \ud z \ud x + I_1' \\
	&= \frac{K_{d,p}}{1-s} \|\nabla u\|^p_{L^p(\T)} + I_1',
\end{align*}
where in the last equality we used the following identity as reported \cite[equation~(3.23)]{Kubin25}:
\begin{equation*}
	(1-s)\int_{B_1} \frac{|\nu\cdot z|}{|z|^{d+sp}}\, \ud z \eqqcolon K_{d,p}, \qquad \forall \nu\in \mathbb{S}^{d-1}.
\end{equation*}
 For \(I_1'\) we write
\[
	I_1' = C \int_{Q_T} \int_{B_1}|z|^{(1-s)p+1-d} \ud z \ud x 
= C |Q_T| d\omega_d\int_0^1 \rho^{(1-s)p} \ud \rho 
= \frac{C'}{p(1-s)+1}.
\]
Multiplying \(I_1'\) by \((1-s)\) and taking the limit as \(s\to1^-\) we get 0. Now we need to show that \((1-s)I_2\) goes to 0 as well:
\begin{align*}
	I_2 &=\int_{Q_T}\int_{B_R\setminus B_1(x)} \frac{|u(x)-u(y)|^p}{|x-y|^{d+sp}}\, \ud y \ud x \leq 2^p \int_{Q_T} \int_{B_R\setminus B_1(x)} \frac{|u(x)|^p+|u(y)|^p}{|x-y|^{d+sp}}\, \ud y \ud x \\
	&= 2^{p} \int_{Q_T}\int_{B_R\setminus B_1(x)} \frac{|u(x)|^p}{|x-y|^{d+sp}}\, \ud y \ud x +2^{p} \int_{B_R}\int_{Q_T\setminus B_1(y)} \frac{|u(y)|^p}{|x-y|^{d+sp}}\, \ud x \ud y \\
	& \leq 2^{p} \int_{Q_T} |u(x)|^p\int_{B_R\setminus B_1} \frac{1}{|y|^{d+sp}}\, \ud y \ud x + 2^p\int_{B_R} |u(y)|^p \int_{B_{T\sqrt{d}}(y)\setminus B_1(y)} \frac{1}{|x-y|^{d+sp}}\, \ud x \ud y \\
	& \leq 2^{p}d\omega_d \| u\|^p_{L^p(\T)} \int_1^R \frac{1}{\rho^{1+sp}} \, \ud \rho +  C 2^{p}d \omega_d  \| u\|^p_{L^p(\T)} \int_1^{T\sqrt{d}} \frac{1}{\rho^{1+sp}}\,\ud \rho \\
	&= \frac{d\omega_d 2^{p}}{sp} (1-R^{-sp})\| u\|^p_{L^p(\T)} +  \frac{C d\omega_d 2^{p}}{sp} \left(1-\bigl(T\sqrt{d}\bigr)^{-sp}\right) \| u\|^p_{L^p(\T)},
\end{align*} 
where \(C=C(R,T,d)\) counts the number of periodic cells intersecting the ball of fixed radius \(R\). As shown above \(I_2\) is bounded in \(s\) as \(s\to 1^-\), therefore \((1-s)I_2\to0\) as \(s\to 1^-\).

\end{proof}
Before stating the \(\Gamma\)-convergence Theorem, we introduce some preliminary results that we will use for compactness. Once again, we refer to \cite{Kubin25} for the proof of the results. To relieve the notation we will denote
\[
B_r(A)\coloneqq \bigcup_{x \in A} B_r(x),
\]
for any set \(A \subset \R^d\) and \(r>0\).

\begin{lemma} \label{lemma: stima kubin 1} For every \(v\in L^p(\T)\), \(h\in \R^d\) and \(0<\rho\leq |h|\) there exists \(C=C(p,d)>0\) such that
\begin{equation*}
    \|\tau_h v-v\|^p_{L^p(\T)} \leq C \frac{|h|^p}{\rho^{d+p}} \int_{B_\rho}\|\tau_y v-v\|^p_{L^p(B_{|h|}(Q_T))}\ud y.
\end{equation*}
\end{lemma}

\begin{lemma} \label{lemma: stime kubin 2}For every \(v\in L^p(\T)\), \(h\in \R^d\) there exists \(C=C(p,d)>0\) such that
\begin{equation*} 
    \|\tau_h v-v\|^p_{L^p(\T)} \leq C (1-s)|h|^{sp} \int_{B_{|h|}}\frac{\|\tau_y v-v\|^p_{L^p(B_{|h|}(Q_T))}}{|y
    |^{d+sp}}\,\ud y.
\end{equation*}
\end{lemma}
To prove this last result, we use the estimate in \Cref{lemma: stima kubin 1} together with those in \cite[Lemma~2.5]{Crismale23} with parameters \(l=ps\). We also recall this result reported in \cite[Theorem~4.26]{Brezis10}:
\begin{theorem}[Fréchet-Kolmogorov] \label{thm: F.K.}
    Let \(p\in[1,+\infty)\) and \((u_n)_n\) be a bounded sequence in \(L^p(\T)\) such that 
    \[
    \lim_{|h|\to0}\sup_{n\in\N} \|\tau_h u_n-u_n\|^p_{L^p(\T)} = 0,
    \]
    then there exists \(u\in L^p(\T)\) and a subsequence \(u_{n_k}\to u\) in \(L^p(\T)\).
\end{theorem}

We also recall this useful characterisation for \(W^{1,p}\) functions:
\begin{lemma}[Equivalent condition for \(W^{1,p}\) functions] \label{lemma: condizione equiv W1p}
Let \(p\in[1,+\infty)\) and \(v\in L^p(\T)\). Then \(v\in W^{1,p}(\T)\) (or \(v\in \mathrm{BV}(\T)\) if \(p=1\)) if and only if there exists \(C>0\) such that for every \(h\in\R^d\)
\[
\|\tau_h v-v\|_{L^p(\T)} \leq C|h|.
\]
\end{lemma}
For a more general result with \(s \in (0,1)\), see \cite[Lemma~6.14]{Leoni23}.

\begin{theorem}[\(\Gamma\)-convergence for \(s\to 1^-\)] \label{thm: G-conv per s a 1}
	Let \(p\in[1,+\infty)\) and \((s_n)_n \subset (0,1)\) be such that \(s_n \rightarrow 1^-\) as \(n \rightarrow +\infty\).
    
		  \emph{1. Compactness.} Let \((u_n)_n\subset L^p(\T)\) such that there exists a constant \(C \in \R\) such that
        \begin{equation} \label{eq: Hyp for compactness 2}
            \sup_{\substack{n\in\N}} \left((1-s_n) \F_p^{s_n}(u_n)+\|u_n\|^p_{L^p(\T)}\right) \leq C,
        \end{equation}
		then, without relabeling subsequences, \(u_n \rightarrow u\) in \(L^p(\T)\) for some \(u \in W^{1,p}(\T)\) if \(p\neq 1\), or \(u\in \mathrm{BV}(\T)\) if \(p=1\).
		
		  \emph{2. \(\Gamma\)-\(\liminf\) inequality.} For all \(u \in L^p(\T)\) and all \((u_n)_n\subset L^p(\T)\) such that \(u_n \rightarrow u\) in \( L^p(\T)\) the following holds:
        \[
        \begin{dcases}
            \liminf_{n\rightarrow +\infty} (1-s_n) \F_p^{s_n}(u_n) \geq K_{d,p}\|\nabla u\|_{L^p(\T)}^p, \qquad & p\neq 1 \\[6pt]
            \liminf_{n\rightarrow +\infty} (1-s_n) \F_1^{s_n}(u_n) \geq K_{d,1}|\mathrm{D}u|(\T), \qquad & p=1.
        \end{dcases}
        \]
        
		\emph{3. \(\Gamma\)-\(\limsup\) inequality.} For all \(u \in L^p(\T)\) there exists \((u_n)_n \subset L^p(\T)\) such that \(u_n \rightarrow u \) in \( L^p(\T)\) and 
        \[
        \begin{dcases}
            \limsup_{n\rightarrow +\infty} (1-s_n) \F_p^{s_n}(u_n) \leq K_{d,p}\|\nabla u\|_{L^p(\T)}^p, &\qquad p\neq 1 \\[6pt]
            \limsup_{n\rightarrow +\infty} (1-s_n) \F_1^{s_n}(u_n) \leq K_{d,1}|\mathrm{D}u|(\T), &\qquad p= 1.
        \end{dcases}
        \]
\end{theorem}

\begin{proof}\strut 

\emph{1. Compactness.} Using \Cref{lemma: stime kubin 2} and the upper bound in \cref{eq: Hyp for compactness 2} we obtain that 
\begin{equation} \label{eq: stima per FK}
     \|\tau_h u_n-u_n\|^p_{L^p(\T)} \leq C|h|^{ps_n}(1-s_n)\F_p^{s_n}(u_n) \leq \tilde C|h|^{ps_n}.
\end{equation}
Using once again the upper bound in \cref{eq: Hyp for compactness 2} and the inequality above, we can apply \Cref{thm: F.K.} to the sequence \((u_n)_n\), hence there exists \(u\in  L^p(\T)\) and a subsequence \((u_{n_k})_k\) such that \(u_{n_k}\rightarrow u\) in \(L^p(\T)\) as \(k\to+\infty\). Furthermore, taking the limit as \(s_n\to 1^-\) in \cref{eq: stima per FK} we can observe that \(u\) satisfies the hypotheses of \Cref{lemma: condizione equiv W1p}, therefore \(u\in W^{1,p}(\T)\) (or \(\mathrm{BV}(\T)\) for \(p=1\)).

\emph{2. \(\Gamma\)-liminf inequality.} For every \(\varepsilon>0\), we start by considering a standard mollifier, \(\rho_\varepsilon:B_\varepsilon\to\R\), and \(u_n^\varepsilon\coloneqq u_n *\rho_\varepsilon\). Using Jensen's inequality, we can easily verify that for every \(s \in (0,1)\)
\begin{align} 
    \F^s_p(u^\varepsilon_n) =& \int_{Q_T}\int_{\R^d} \frac{|u_n^\varepsilon(x)-u_n^\varepsilon(y)|^p}{|x-y|^{d+sp}} \, \ud y \ud x \label{eq: convoluzione abbassa energia}\\
    \leq & \int_{Q_T}\int_{\R^d}\int_{B_\varepsilon} \frac{|u_n(x-z)-u_n(y-z)|^p}{|x-y|^{d+sp}} \rho_\varepsilon(z) \, \ud z \ud y \ud x \nonumber\\
    = & \F^s_p(u_n). \nonumber
\end{align}
We now prove that for every \(\varepsilon>0\) and \(R>0\) it holds 
\begin{equation} \label{eq: disug dimostrazione liminf}
    K_{d,p} \liminf_{n\to+\infty} \int_{Q_T}|\nabla u_n^\varepsilon|^p \left(\text{dist}(x,\partial B_R)\right)^{p(1-s_n)} \ud x \leq \liminf_{n\to+\infty} (1-s_n)[u_n^\varepsilon]^p_{s_n,p}|_{B_R}.
\end{equation}
Indeed, for every \(n\), \cref{eq: sviluppo alla p} holds true with a different constant \(C_n=C\|\mathrm{D^2}u^\eps_n\|_{\infty}\), i.e.
\[
	|u_n^\varepsilon(x)-u_n^\varepsilon(y)|^p= |\nabla u_n^\varepsilon(x)\cdot(x-y)|^p + C_n\left(|x-y|^{p+1}\right),
\]
and thus, thanks to \cref{eq: Hyp for compactness 2} and the following estimate
\[
	|\mathrm{D}^2 u_n^\varepsilon(x)| = |(u*\mathrm{D}^2\rho_\varepsilon)(x)| \leq \|u_n\|_{L^1(\T)}\|\rho_\varepsilon\|_{C^2},
\]
(see \cref{rmk: inclusioni sobolev}), it holds true uniformly in \(n\) for some \(C>0\):
\[
	|u_n^\varepsilon(x)-u_n^\varepsilon(y)|^p \geq |\nabla u_n^\varepsilon(x)\cdot(x-y)|^p - C|x-y|^{p+1}.
\]
Let \(x\in Q_T\), setting \(\delta\coloneqq \text{dist}(x,\partial B_R)\) we get
\begin{align}
    (1-s_n)\int_{B_R} \frac{|u_n^\varepsilon(x)-u_n^\varepsilon(y)|^p}{|x-y|^{d+s_np}}\, \ud y \geq& (1-s_n)\int_{B_\delta(x)} \frac{|u_n^\varepsilon(x)-u_n^\varepsilon(y)|^p}{|x-y|^{d+s_np}}\, \ud y \nonumber\\
    \geq& (1-s_n)\int_{B_\delta(x)} \left\lvert \nabla u_n^\varepsilon(x)\cdot (x-y) \right\rvert^p |x-y|^{-d-s_np} \ud y \nonumber\\
    &- (1-s_n)\int_{B_\delta(x)} C\left(|x-y|^{p(1-s_n)+1-d}\right) \ud y \nonumber \\
    \geq& K_{d,p}\delta^{p(1-s_n)} |\nabla u_n^\varepsilon(x)|^p - (1-s_n)C(\varepsilon,R,d), \label{eq: conto liminf s a 1}
\end{align}
where once again we used \cite[equation~(3.23)]{Kubin25}. By integrating \cref{eq: conto liminf s a 1} over \(Q_T\) and taking the \(\liminf\) for \(n\to+\infty\) we recover inequality \eqref{eq: disug dimostrazione liminf}. Now we use \cref{eq: convoluzione abbassa energia} on the right side of \cref{eq: disug dimostrazione liminf} and we take \(R\) such that \(\text{dist}(Q_T,\partial B_R)>1\). Recall that \(u_n^\varepsilon \to u^\varepsilon\) in \(L^p(\T)\) and that, thanks to \cref{eq: conto liminf s a 1} and the upper bound in \cref{eq: Hyp for compactness 2}, \((\nabla u_n^\varepsilon)_n\) is equibounded in \(L^p(\T)\). This means that we can apply Fatou's lemma to get:
\begin{align*}
    \liminf_{n\to+\infty} (1-s_n)[u_n]^p_{s_n,p}|_{B_R} &\geq \liminf_{n\to+\infty} (1-s_n)[u_n^\varepsilon]^p_{s_n,p}|_{B_R} \\
    &\geq K_{d,p} \liminf_{n\to+\infty} \int_{Q_T}|\nabla u_n^\varepsilon|^p \left(\text{dist}(x,\partial B_R)\right)^{p(1-s_n)} \ud x \\
    &\geq K_{d,p} \liminf_{n\to+\infty} \int_{Q_T}|\nabla u_n^\varepsilon|^p \ud x  \\
    & \geq K_{d,p} \int_{Q_T}|\nabla u^\varepsilon|^p \ud x.
\end{align*}
Taking the limit for \(R\to+\infty\) and \(\varepsilon\to0\), recalling that \(u^\varepsilon\to u\) in \(W^{1,p}(\T)\) (or \(\mathrm{BV}(\T)\) for \(p=1\)) we have the thesis.

\emph{3. \(\Gamma\)-limsup inequality.} In the same way as we did for \(s\to0^+\) we show the limit for \(u\in C^\infty(\T)\), using the pointwise convergence with recovery sequence \(u_n\equiv u\) and the density argument in \cite[remark~1.29]{Braides02}.

\end{proof}

\section{Characterisation of minimisers for \(\F^s_p\) and \(\F^0_p\)}\label{sez: minimi rigidi}

In this Section we aim to minimise $\mathcal F^s_p$ and its $\Gamma$-limit \(\F^0_p\), among the functions $u$ whose distributional derivative is an atomic measure, screened by the Lebesgue measure. 
The main motivation for our analysis comes from Materials Science and, in particular, from the study of misfit dislocations at semi-coherent interfaces \cite{van50}.

\subsection{Elastic energy on the half-plane}

In this setting, we consider \(p=2\), \(s=\frac{1}{2}\) and two square atomic lattices of slightly different atomic spacing meeting at a flat interface. Then, one expects that, along the interface, a periodic configuration of dislocations, i.e. topological singularities of the elastic strain, should arise. Since the \(H^\frac{1}{2}\) energy induces a repulsive logarithmic-type interaction between defects \cite{Garroni12}, a natural question is whether their optimal arrangement is given by the universal triangular Abrikosov lattice, as is standard for logarithmically interacting particles \cite{Sandier12}, or whether it is strongly constrained by the square symmetry of the underlying atomic lattices. Regardless of the underlying motivation, our objective is showing that, in a simplified one-dimensional model, the minimiser of the \(W^{s,p}\)-seminorm energy shows indeed a fundamental periodicity, independent of the \(T\)-periodicity assumption.

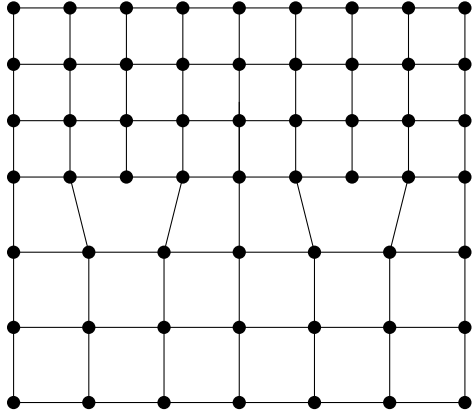
\begin{figure}[h]
	\centering
\makebox[0pt]{
\scalebox{1}{
	\begin{tikzpicture}[
		punto/.style={circle, fill=black, inner sep=0, minimum size=5pt},
		freccia_centrale/.style={
			decoration={
				markings,
				mark=at position 0.5 with {\arrow{Stealth[length=3mm, width=2.5mm]}}
			},
			postaction={decorate}
		}
		]
		\draw (0,2)--(0,3);
		\draw (1,2)--(3/4,3);
		\draw (2,2)--(2+1/4,3);
		\draw (3,2)--(3,4);
		\draw (4,2)--(3+3/4,3);
		\draw (5,2)--(5+1/4,3);
		\draw (6,2)--(6,3);

		\def\colonnea{7}      
		\def\righea{3}        
		\def\passoXa{1}    
		\def\passoYa{1}    
		
		\pgfmathtruncatemacro{\maxXa}{\colonnea - 1}
		\pgfmathtruncatemacro{\maxYa}{\righea - 1}
		
		
		\foreach \j in {0,...,\maxYa} {
			\draw (0, \j * \passoYa) -- (\maxXa * \passoXa, \j * \passoYa);
		}
		
		\foreach \i in {0,...,\maxXa} {
			\draw (\i * \passoXa, 0) -- (\i * \passoXa, \maxYa * \passoYa);
		}
		
		\foreach \i in {0,...,\maxXa} {
			\foreach \j in {0,...,\maxYa} {
				\coordinate (p-\i-\j) at (\i * \passoXa, \j * \passoYa);
				\node[punto] at (\i * \passoXa, \j * \passoYa) {};
			}
		}
		
		\def\colonneb{9}      
		\def\righeb{4}        
		\def\passoXb{3/4}    
		\def\passoYb{3/4}    
		\def\trasly{\righea*\passoYa} 
		
		\pgfmathtruncatemacro{\maxXb}{\colonneb - 1}
		\pgfmathtruncatemacro{\maxYb}{\righeb - 1}
		
		
		\foreach \j in {0,...,\maxYb} {
			\draw (0, \trasly +\j * \passoYb) -- (\maxXb * \passoXb,\trasly + \j * \passoYb);
		}
		
		\foreach \i in {0,...,\maxXb} {
			\draw (\i * \passoXb, \trasly) -- (\i * \passoXb, \maxYb * \passoYb+ \trasly);
		}
		
		\foreach \i in {0,...,\maxXb} {
			\foreach \j in {0,...,\maxYb} {
				\coordinate (q-\i-\j) at (\i * \passoXb, \j * \passoYb + \trasly);
				\node[punto] at (\i * \passoXb, \j * \passoYb + \trasly) {};
			}
		}
	\end{tikzpicture}
}
}
    \label{fig: edge dislocations}
	\caption{A visualisation of edge dislocations in a square crystal lattice with two different atomic distances. The image represents the view of a section orthogonal to the \(x_3\) axis.}
\end{figure}%

Following the general idea of \cite[Section 2]{DeLuca24}), let $\mathbb{H}^\pm:=\R\times\R^\pm\times\R$ and $\mathbb{H}^0:=\R\times\{0\}\times\R$. According to isotropic linear elasticity, the elastic energy of a planar displacement $U^\pm\in H^1(\mathbb{H}^\pm,\R^3)$ (where \(U\coloneqq U^\pm\) on \(\mathbb{H}^\pm\)) is given by 
\begin{equation*}
    \mathcal{E}_{3\mathrm{D}}^\mathrm{el}(U^\pm) \coloneqq  \frac{1}{2}\int_{\mathbb{H}^\pm} \lambda \left|\text{tr}(e(U))\right|^2 +2\mu \left|e(U) \right|^2 \ud x= \frac{1}{2}\int_{\mathbb{H}^\pm} \sigma(U):e(U) \,\ud x,
\end{equation*}
where $\lambda$ and $ \mu$ represent the Lamé moduli of the elastic material, satisfying $\mu>0$, $\lambda+\mu>0$ and \(e(U)=\tfrac{1}{2}\left(\nabla U + \nabla^\mathrm{T} U\right)\) and \(\sigma(U)=\lambda \mathrm{tr}(e(U))\mathrm{Id}+2\mu e(U)\). 

For modelling reasons, we consider edge dislocations as parallel lines lying on the \(\{x_2=0\}\) plane and orthogonal to \(e_1\). Due to the symmetry of the lattice, we can apply some simplifications to the model by considering only displacements in the directions orthogonal to the dislocation lines, i.e. the displacement vector is \(U=(U_1,U_2,0)\) and \(U=U(x_1,x_2)\). Following standard assumptions for straight edge dislocations (see, e.g., \cite[equation (1.1)]{cinesi20}), we impose mirror symmetry with respect to the slip plane. Specifically, the shear displacement \(U_1\) is assumed to be odd in \(x_2\), while the vertical displacement \(U_2\) is even. Now, denoting with \(\mathsf{H}^\pm:=\R\times\R^\pm\) and \(\mathsf{H}^0:=\R\times\{0\}\), we take the trace of \(U^\pm\in H^1(\mathsf{H}^\pm,\R^2)\) on \(\mathsf{H}^0\) so that \(U^\pm(x_1,0)=(u_1^\pm(x_1),u_2^\pm(x_1))=u^\pm(x_1)\in H^{\frac{1}{2}}(\mathsf{H}^0,\R^2)\). By virtue of the aforementioned symmetries, the traces satisfy \(u_1^+(x_1) = -u_1^-(x_1)\) and \(u_2^+(x_1) = u_2^-(x_1)\). From now on we consider the 2D elastic energy:
\begin{equation} \label{eq: 2D energia elastica}
    \mathcal{E}_{2\mathrm{D}}^{\mathrm{el}}(U^\pm) \coloneqq  \frac{1}{2}\int_{\mathsf{H}^\pm} \lambda \left|\text{tr}(e(U))\right|^2 +2\mu \left|e(U) \right|^2 \ud x= \frac{1}{2}\int_{\mathsf{H}^\pm} \sigma(U):e(U) \,\ud x.
\end{equation}
We now observe that the displacement vector that minimises the energy on \(\mathsf{H}^\pm\) given the trace \(u^\pm\in H^{\frac{1}{2}}(\mathsf{H}^0,\R^2)\), i.e.
\[
    \mathrm{argmin}\left\{\mathcal{E}_{2\mathrm{D}}^{\mathrm{el}}(U^\pm)\,|\, U^\pm|_{\mathsf{H}^0}=u^\pm\right\},
\]
solves the Euler-Lagrange equations
\begin{equation} \label{eq: div}
    \begin{cases}
        \text{Div}(\sigma(U))=0 \quad & \text{on }\mathsf{H}^\pm \\
        U^\pm=u^\pm \quad & \text{on }\mathsf{H}^0,
    \end{cases}
\end{equation}
where we denote by \virgolette{\(\text{div}\)} the usual divergence operator and with \virgolette{\(\text{Div}\)} the divergence operator for matrix fields. By the symmetry of the stress tensor, we have \(\sigma(U):e(U) = \sigma(U):\nabla U\), and applying the Leibniz rule to the divergence of the vector field \(\sigma(U)U\), we derive the pointwise identity
\[
    \text{div}(\sigma(U)U) = \text{Div}(\sigma(U)) \cdot U + \sigma(U):\nabla U = \text{Div}(\sigma(U)) \cdot U + \sigma(U):e(U).
\]
Using this relation on \cref{eq: 2D energia elastica}, we can write
\[
    \mathcal{E}_{2\mathrm{D}}^{\mathrm{el}}(U^\pm)= \frac{1}{2}\int_{\mathsf{H}^\pm} \text{div}(\sigma(U)U)-\text{Div}(\sigma(U))\cdot U \,\ud x,
\]
where the second term is 0 thanks to \cref{eq: div}. Now, we apply the Divergence Theorem to the first term and we get
\[
    \mathcal{E}_{2\mathrm{D}}^{\mathrm{el}}(U^\pm)=\frac{1}{2}\int_{\mathsf{H}^\pm} \text{div}(\sigma(U)U)\, \ud x = \frac{1}{2}\int_{\mathsf{H}^0} \sigma(U^\pm)U^\pm\cdot\nu^\pm \,\ud x_1 = \frac{1}{2}\int_{\mathsf{H}^0} \sigma_{12}(U^\pm) u_1^\pm + \sigma_{22}(U^\pm) u_2^\pm\,\ud x_1,
\]
where \(\nu^\pm=\mp e_2\). To conclude, we recall that the symmetry assumptions on \(U\) (\(U_1\) odd and \(U_2\) even with respect to \(x_2\)) imply that the strain components \(e_{11}\) and \(e_{22}\) are both odd in \(x_2\). As a result, the normal stress \(\sigma_{22}(U^\pm)\) is an odd function as well. Assuming continuity in \(x_2\) of \(\sigma_{22}(U)\) for modelling reasons, we deduce that \(\sigma_{22}(U)=0\) on \(\mathsf{H}^0\) (in accordance with \cite[equation (2.48)]{cinesi20}). Therefore, the term involving \(u_2^\pm\) vanishes, and we can rewrite the remaining energy as
\[
     \mathcal{E}_{2\mathrm{D}}^{\mathrm{el}}(U) = \mathcal{E}_{2\mathrm{D}}^{\mathrm{el}}(U^+)+\mathcal{E}_{2\mathrm{D}}^{\mathrm{el}}(U^-) = C(\lambda,\mu) \int_{\R} \int_{\R} \frac{|u_1(x)-u_1(y)|^2}{|x-y|^{2}} \,\ud y \ud x,
\]
where \(u_1\coloneqq u_1^+-u_1^-\). In hindsight, we can assume that \(u\) is a scalar valued function to get the same result.

Assume now that \(u\in H^\frac{1}{2}_{\mathrm{loc}}(\R)\) is \(T\)-periodic. Let \(U\in H^1([0,T]\times\R,\R^2)\) be the unique minimiser of the elastic energy 
\[
    \mathcal{E}_{T}^{\mathrm{el}}(U) \coloneqq  \frac{1}{2}\int_{[0,T]\times\R} \lambda \left|\text{tr}(e(U))\right|^2 +2\mu \left|e(U) \right|^2 \ud x= \frac{1}{2}\int_{[0,T]\times\R} \sigma(U):e(U) \,\ud x,
\]
among functions \(V \in H^1([0,T]\times\R,\R^2)\) such that \(V|_{[0,T]\times\{0\}}=u\) and \(V(0,\cdot)=V(T,\cdot)\) in the sense of traces. Moreover, let \(\tilde U\in H^1_{\mathrm{loc}}(\R^2,\R^2)\) the \(T e_1\)- periodic extension of \(U\). Then \(\tilde U\) solves \cref{eq: div} and, by similar computations, using periodicity we get 
\[
    \mathcal{E}_{T}^{\mathrm{el}}(\tilde U) = C(\lambda,\mu,T) \F^\frac{1}{2}_2(u).
\]

\subsection{Generalised \(d\)-dimensional energy}

There is no immediate natural extension of this model to the \(d\)-dimensional setting, or for \((s,p)\neq\left(\frac{1}{2},2\right)\); nevertheless, the problem given by the energy \cref{eq: seminorma sp} is mathematically intriguing. Following the idea of the linear elasticity model, we consider the ambient space $\mathbb{R}^d$ with coordinates $x = (x_1, x_2, x')$, where $x' = (x_3, \dots, x_d) \in \mathbb{R}^{d-2}$ represent the directions parallel to the defect \virgolette{lines}. The domains are suitably generalized as the half-spaces $\mathbb{H}^\pm := \mathbb{R} \times \mathbb{R}^\pm \times \mathbb{R}^{d-2}$, separated by the hyperplane $\mathbb{H}^0 := \mathbb{R} \times \{0\} \times \mathbb{R}^{d-2}$.

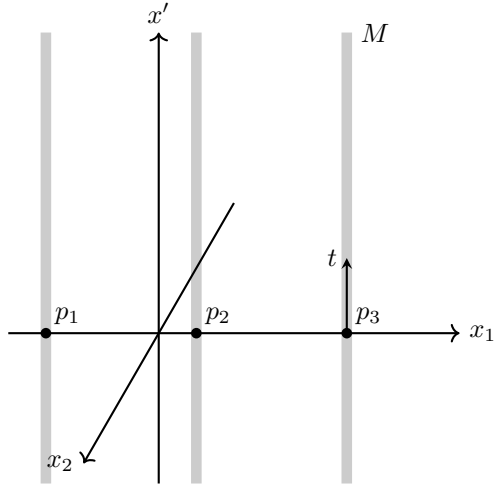
\begin{figure}[H]
	\centering
	\begin{tikzpicture}[
		x={(1cm,0cm)}, 
		z={(0cm,1cm)}, 
		y={(-120:0.5cm)} 
		]
		
		\draw[line width=4pt, black!20](-1.5,0,-2) -- (-1.5,0,4);
		\draw[line width=4pt, black!20](.5,0,-2) -- (.5,0,4);
		\draw[line width=4pt, black!20](2.5,0,-2) -- (2.5,0,4) node[black,right]{\(M\)};
		
		\draw[->, thick] (-2,0,0) -- (4,0,0) node[right] {$x_1$};
		\draw[->, thick] (0,-4,0) -- (0,4,0) node[left] {$x_2$};
		\draw[->, thick] (0,0,-2) -- (0,0,4) node[above] {$x'$};
		
		\fill (-1.5,0,0) circle (2pt) node[above right]{\(p_1\)};
		\fill (.5,0,0) circle (2pt) node[above right]{\(p_2\)};
		\fill (2.5,0,0) circle (2pt) node[above right]{\(p_3\)};
		
		\draw[black, thick, -{stealth}] (2.5,0,0) -- (2.5,0,1) node[left]{\(t\)};
	\end{tikzpicture}
	\label{fig: dislocazioni in dim d}
	\caption{In the scheme above we picture dislocation as \virgolette{lines} orthogonal to the \((x_1,x_2)\)-plane in the direction \(x'\), representing the last \(d-2\) variables. In this setting the multivector \(t\) tangent to the dislocation manifold \(M\) is pointing in the direction of \(x'\).}
\end{figure}%

In this framework, the deformation of the crystal is modelled by the absolutely continuous part of the distributional displacement gradient, represented by a $d \times d$ matrix field $\beta \in L^2(\R^d,\R^{d\times d})$. The presence of an array of parallel edge dislocations introduces a topological constraint on $\beta$. Due to the aforementioned modelling assumptions, it is reasonable to assume the displacement vector \(U\in H^1(\R^d\setminus M,\R^d)\), with \(U^\pm=U|_{\mathbb{H}^\pm}\), to be regular outside of the defect manifold \(M\coloneqq P\times \R^{d-2}\subset \mathbb{H}^0\), where \(P=\bigcup_{k}\{p_k\}\times\{0\}\subset \R\times\{0\}\).

The circulation of $\beta$ around any defect must be equal to a fixed vector that we take to be \(e_1\); therefore, in the sense of distributions, this translates to the curl constraint (denoting with \virgolette{Curl} the curl operator for matrix fields):
\begin{equation}\label{eq: curl constraint}
    \mathrm{Curl}(\beta) = e_1 \otimes t \otimes \mu,
\end{equation}
where $t=e_3 \wedge \ldots \wedge e_d$ is the tangent multivector to the defect manifold \(M\), and \(\mu(x_1,x_2) \coloneqq \left(\sum_{k}\delta_{p_k}(x_1)\right)\otimes \delta_0(x_2)\) is a measure supported on \(P\). 

Let $\beta^{(i)}$ denote the $i$-th row of $\beta$, viewed as a vector field in $\mathbb{R}^d$. The standard curl of a vector field in $\mathbb{R}^d$ is the antisymmetric tensor defined by $\mathrm{curl}(\beta^{(i)})_{jk} = \partial_j \beta_{ik} - \partial_k \beta_{ij}$. Because of the curl constraint, only the first row carries a topological singularity. Furthermore, since the defect lines are orthogonal to the $\left(x_1,x_2\right)$-plane, this singularity is entirely confined to the $e_1 \wedge e_2$ subspace. Thus, \cref{eq: curl constraint} reduces to:
\begin{equation} \label{eq: rotore_esplicito}
    \mathrm{curl}(\beta^{(i)}) = 
    \begin{cases}
    \mu(x_1, x_2) \, (e_1 \otimes e_2 - e_2 \otimes e_1) & \quad \text{if } i = 1, \\
    0 & \quad \text{if } i \ge 2.
    \end{cases}
\end{equation}

We are interested in minimising the Dirichlet \(p\)-energy of the displacement. Since the measure \(\mu\) is translation-invariant along the \(x'\) directions, we can search for minimisers that share this symmetry, i.e. \(\beta = \beta(x_1, x_2)\). Consequently, we minimise the energy per unit volume along the defect lines, which corresponds to the integrated energy on the transversal section:
\begin{equation} \label{eq: energia farlocca}
    \mathcal{E}^2_{\mathrm{2D}}(U) \coloneqq \int_{\mathbb{R}^2} |\beta(x_1, x_2)|^2\ud x_1 \ud x_2 = \int_{\mathsf{H}^+} |\nabla U^+|^2 \ud x_1 \ud x_2 + \int_{\mathsf{H}^-} |\nabla U^-|^2 \ud x_1 \ud x_2 
\end{equation}
subject to the linear constraint in \cref{eq: rotore_esplicito}, where \(\mathsf{H}^\pm\) and \(\mathsf{H}^0\) are defined as in the previous model.

The structure of the minimiser greatly simplifies thanks to the properties of the energy and the constraint. Because $\beta$ depends only on $x_1$ and $x_2$, all derivatives with respect to $x_k$ (for $k \ge 3$) identically vanish. Looking at \cref{eq: rotore_esplicito}, the condition $\mathrm{curl}(\beta^{(i)})_{jk} = 0$ for $k \ge 3$ and $j \in \{1,2\}$ yields \(\partial_j \beta_{ik} = 0\). This implies that all components in the columns $3, \ldots, d$ of $\beta$ are constant. To yield a finite $L^2$ energy on the cross section, these constants must be zero. Hence, the displacement has no gradient along the defect lines. 

Finally, \cref{eq: rotore_esplicito} shows that the topological constraint only acts on the first row $\beta^{(1)}$. Any non-zero component in the rows $\beta^{(i)}$ for $i \ge 2$ would strictly increase the energy. By the strict convexity of the integrand $|\beta|^2$, the unique minimiser must have $\beta^{(i)} \equiv 0$ for all $i \ge 2$. In this case, the underlying displacement vector simplifies to a purely scalar planar shifts: $U = (U_1(x_1, x_2), 0, \dots, 0)$.

\begin{remark}
    If we consider the linear isotropic elastic energy (as in \cref{eq: 2D energia elastica}), the symmetric strain tensor $e(\beta) = \frac{1}{2}(\beta + \beta^\mathrm{T})$ couples the first and second rows (i.e., $\beta_{12}$ and $\beta_{21}$). Consequently, setting $\beta_{21} = 0$ would break the symmetry of $e_{12}$. Thus, $\beta_2$ (and the vertical displacement $U_2$) does not vanish.
\end{remark}

Once again we observe that minimising the energy in \cref{eq: energia farlocca}, with trace \(U^\pm|_{\mathsf{H}^0}=u^\pm(x_1)\in H^{\frac{1}{2}}(\mathsf{H}^0\setminus P)\) leads to the following Euler-Lagrange equations:
\begin{equation} \label{eq: EL farlocco}
    \begin{cases}
    \Delta U =0 \quad & \text{on }\mathsf{H}^\pm \\
    U^\pm=u^\pm \quad & \text{on }\mathsf{H}^0.
    \end{cases}
\end{equation}
Following the same modelling assumptions introduced for linear elasticity, we impose a mirror anti-symmetry with respect to the slip plane \(\mathsf{H}^0\), such that $U_1(x_1, x_2) = -U_1(x_1, -x_2)$. By the assumed symmetry, $u^+(x_1) = -u^-(x_1)$. The total slip (or displacement jump) across the interface is given by $u(x_1) \coloneqq u^+(x_1) - u^-(x_1) = 2u^\pm(x_1)$. By a trivial integration by parts, using \cref{eq: EL farlocco}, we write the energy as 
\[
    \mathcal{E}^2_{\mathrm{2D}}(U) = \int_{\mathsf{H}^0} u^+\partial_{\nu^+} U^+ \ud x_1 + \int_{\mathsf{H}^0} u^-\partial_{\nu^-} U^- \ud x_1,
\]
where \(\nu^\pm = \mp e_2\). Observe that, given the hypothesis on \(u\), the energy above is not well-defined. To overcome this issue, we can apply a standard regularisation procedure as we will do in \Cref{sez: minimi rigidi} for the critical and super-critical cases.
Provided this additional hypothesis, we can take \(U^\pm \in H^1(\mathsf{H}^\pm)\) and \(u \in H^{\frac{1}{2}}(\mathsf{H}^0)\). Observe that \(\partial_\nu U\) is the Dirichlet-to-Neumann map as claimed in \cite{Caffarelli07}, therefore using the symmetry assumptions on \(u^\pm\) we get the Gagliardo seminorm energy:
\[
    \mathcal{E}^2_{\mathrm{2D}}(U) = 2\int_{\R}\int_{\R} \frac{|u(x)-u(y)|^2}{|x-y|^2} \,\ud y \ud x.
\]

\begin{remark}
	If we instead take \(U\in W^{1,p}(\R^d\setminus M,\R^d)\) and the energy 
	\[
		\mathcal{E}^{p}_{d\mathrm{D}}(U)=\int_{\R^d} |\beta|^p \ud x,
	\]
	once again we can reduce our attention to the 2-dimensional energy density and \(U(x_1,x_2)\) solves 
	\[
		\begin{cases}
			(-\Delta)_p U =0 \quad & \text{on }\mathsf{H}^\pm \\
			U^\pm=u^\pm \quad & \text{on }\mathsf{H}^0,
		\end{cases}
	\]
	with trace datum \(u^\pm\in W^{1-\frac{1}{p},p}(\R\setminus P)\). Alas, we observe that for \(p\neq 2\) the trace operator is not an isometric operator (up to a constant, with respect to the Sobolev seminorm) on \(p\)-harmonic extensions of \(W^{1-\frac{1}{p},p}\) boundary data on the half-plane. The trace energy and \(\mathcal{E}^p_{2\mathrm{D}}\) are, in general, only equivalent.
\end{remark}

As for the elastic energy, one can consider displacements \(U^\pm\in H^1([0,T]\times \R^\pm \times \R^{d-2}\setminus M,\R^d)\) where, in this setting, the defects manifold is \(M=P\times \R^{d-2}\) and \(P\) is \(T\)-periodic. Then, with the same argument as before the displacement is independent of the last \(d-2\) variables and \(U=(U_1(x_1,x_2),0,\ldots,0)\). Now we consider a trace \(u\in H^{\frac{1}{2}}_{\mathrm{loc}}(\R\setminus P)\) \(T\)-periodic, and \(U\) the unique minimiser of the energy (finite up to regularisation)
\[
	\mathcal{E}^{2}_{T}(U) \coloneqq \int_{[0,T]\times\R} |\nabla U|^2 \ud x,
\]
among the functions \(V\in H^1([0,T]\times\R \setminus M,\R^2)\) with \(V|_{[0,T]\times\{0\}}=u\) and \(V(0,\cdot)=V(T,\cdot)\) in the sense of traces. Now let \(\tilde U\) be the \(Te_1\)-periodic extension  of \(U\), then \(\tilde U\) solves \cref{eq: EL farlocco} and, by similar computations, using periodicity we get

\[
	\mathcal{E}^{2}_T(\tilde U) = C(T) \F^\frac{1}{2}_2(u).
\]
\subsection{Configurations and rigid variations}
As previously discussed the interface is characterised by an elastic part, where the distributional derivative of $u$ is constant, say 1, and a plastic part, corresponding to dislocations, that are, by their discrete nature, quantised. Assuming a macroscopic periodicity $T \in \N$, we require the presence of exactly $T$ dislocations within the period $[0,T)$. Our problem therefore consists in minimising the energy $\mathcal{F}^s_p$ among $T$-periodic functions $u$ whose distributional derivative satisfy the constraint
\begin{equation} \label{eq: condizione rigidità}
    \mathrm{D}u = \mathcal{L}^1 - \sum_{i=1}^{T}\delta_{x_i}, \qquad \text{in } [0,T).
\end{equation}

Two important comments are in order. First, due to the presence of the Dirac deltas, $\mathcal{F}^s_p(u) = +\infty$ whenever $sp \geq 1$. Consequently, our analysis will initially focus on the sub-critical regime $sp < 1$, where the energy is finite. The critical and super-critical cases ($sp \geq 1$) will be addressed later through a suitable regularisation procedure. Second, the minimality of the 1-periodic configuration (i.e., with equispaced singularities) has already been established in \cite{Goldman25} for the physical case $p=2$ and \(s\in (0, \frac{1}{2})\). In this paper, we extend it to the full range $p \in [1, +\infty)$ and \(s \in (0,\frac{1}{p}]\), and to \(s\in (\frac{1}{p},1)\) adding further hypotheses. Ultimately, we will show that the equispaced configuration is the unique minimiser (up to translations), demonstrating in particular that the artificial $T$-periodicity assumption does not play any role.

Since any function $u$ satisfying \cref{eq: condizione rigidità} is entirely determined, up to an additive constant, by the position of its singularities $x_i$, the energy $\mathcal{F}^s_p$ is invariant under vertical translations of $u$ and depends exclusively on the position of these points. 

To begin, we need to define the proper space in which we set our minimisation problem.

\begin{definition}[Configuration] \label{def: config}
Let \(T\in\N\). We define the space of \(T\)-periodic \emph{configurations} as 
\begin{equation*}
    \X_T\coloneqq \left\{\bold{X}= \bigcup_{k\in\Z}\left(\bold{\bar{X}}+kT\right) \;\left|\;\bold{\bar{X}}=(x_1,\ldots,x_T)\in[0,T)^T,\,x_i\leq x_j \text{ if } i<j\right.\right\}.
\end{equation*}
Furthermore, we call \(0\leq \ud[\bold{X}]\coloneqq \min\{|x_i-x_j| \,:\, x_i,x_j\in\bold{X}\}\), the minimum distance between two points in \(\bold X\), and we say that \(\bold{X}\) is a \(l\)-\emph{regular configuration}, for some \(l\in(0,1]\), if \(l\leq \ud[\bold X]\). We denote the space of such configurations as \(\X^l_T\). Moreover, we will denote by \(\X^0_T\coloneqq \bigcup_{l>0}\X^l_T\) the space of configurations with non-overlapping jumps (denoted simply by \emph{regular configurations}).
\end{definition}

\begin{remark} 
    Let us observe that, by convention, we will index the elements in \(\bold{X}=(x_n)_n\) as follows: \(x_1\) is the first element of \(\bold{\bar X}\) and for every \(n\in\Z\), \(x_n\coloneqq x_{i}+k T\), where \(i\in\{0,\ldots, T-1\}\) and \(k\in\Z\) are the unique numbers such that \(n=i+kT\). The reason why we consider \(x_i\in[0,T)\) instead of \([0,T]\) is to avoid overlapping during the periodicisation process. Indeed, we observe that the identification \(\R/T\Z \cong \mathbb{S}^1 \cong [0,T)\) holds. 
\end{remark}

\begin{definition}[Energy of a configuration]
    Let \(\bold{X}\in\mathbb{X}_T\), we say that the energy of this configuration is 
    \[
    \F^s_p[\bold{X}] \coloneqq \F^s_p(u[\bold{X}]),
    \]
    where \(u[\bold{X}]\) is the \(T\)-periodic function defined (up to vertical translations) by the relation 
    \[
        \mathrm{D}u[\bold{X}]= \mathcal{L}^1-\sum_{x_i\in\bold X} \delta_{x_i} \eqqcolon \mathcal{L}^1 - \mu^{\bold X}.
    \]
\end{definition}
We observe that the energy \(\F^s_p[\bold{X}]=\F^s_p[\bold{X}+a]\), for every \(a\in\R\), where \(\bold{X}+a = (x_n+a)_n\). We also note that the energy is invariant under vertical translations, which implicitly justifies the notation \(\F^s_p[\bold{X}]=\F^s_p(u[\bold{X}])=\F^s_p(u[\bold{X}]+c)\), for every \(c\in\R\).
Our goal will be to prove that the unique minimiser for the functional \(\F^s_p\) in the set \(\X_T\) is the equispaced configuration and that the same holds true for \(\F^0_p\).

\subsection{The sub-critical case: \(sp<1\)}
In this Section, we will begin by considering \textit{regular configurations}, showing that the second variation of the energy is positive definite on them and, in particular, that the equispaced configuration is the only critical point among them. Following this result we will prove that all critical points are indeed regular configurations and therefore the equispaced configuration is the unique global minimiser for the energy. To achieve this goal we need to find a suitable class of variations on the aforementioned class of configurations. The natural choice in the space \(\X_T\) is to move a point \(x_i\in\ \bold X\) by a small quantity \(h>0\). We call this operation \emph{rigid variations}.

\begin{definition}[Rigid variation] \label{def: var rigida}
Let \(0<h< h_0\) be small, \(\bold{X}\in\X^{h_0}_T\) and \(x_i \in \bold X\) be fixed. Let \(I_h\coloneqq [x_i,x_i+h)\) and \(\mathcal{I}_h\coloneqq\bigcup_{k\in\Z} \left(I_h+kT\right)\) the union of the intervals \(I_h\) which have been \(T\)-translated on \(\R\). A rigid variation of \(\bold X\) in the point \(x_i\) is defined as 
\[
	\bold X + h e_i \coloneqq (x_1, \ldots, x_i+h, \ldots, x_T)+kT, \qquad \text{ for all } k \in \Z.
\]
Moreover, we observe that the effect of a variation of \(\bold{X}\) on the corresponding function \(u[\bold{X}]\) is 
\[
u[\bold{X}+he_i](x)= u[\bold{X}](x)+ \rchi_{\I_h}(x).
\]
We can give an appropriate definition also for \(-h_0<h<0\) changing the sign of the characteristic function.
\end{definition}
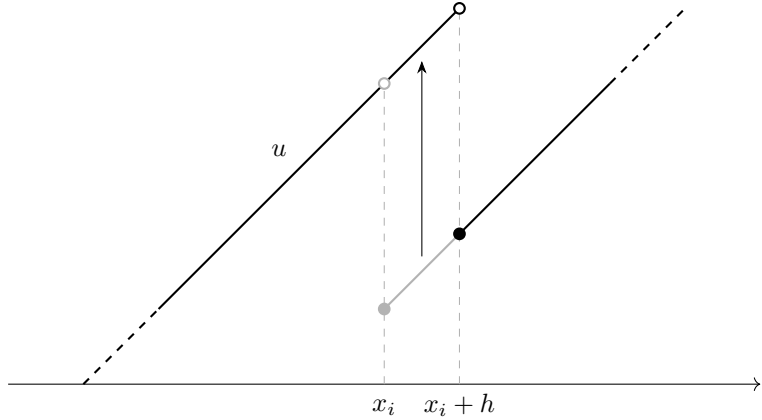
\begin{figure}
	\centering
	\begin{tikzpicture}
		\draw[->] (-5,0)--(5,0);
		
		\draw[dashed,black!30] (0,0) node[black,below=.08cm]{\(x_i\)} -- (0,4);
		\draw[dashed,black!30] (1,0) node[black,below]{\(x_i+h\)} -- (1,5);
		
		\draw[thick] (-3,1)--(1,5) node[pos=.4,above=.3cm]{\(u\)};
		\draw[dashed,thick] (-4,0)--(-3,1);
		\draw[black!30,fill=white,thick] (0,4) circle (2pt);
		\draw[fill=white,thick] (1,5) circle (2pt);
		
		\draw[black!30,thick]	(0,1)--(1,2);
		\draw[thick]	(1,2)--(3,4);
		\draw[dashed,thick]	(3,4)--(4,5);
		\filldraw[black!30,thick] (0,1) circle (2pt);
		\filldraw[thick] (1,2) circle (2pt);

		\draw[arrows={-Stealth}] (.5,1.7) -- (.5,4.3);
	\end{tikzpicture}
    \label{fig: schema variazione rigida}
    \caption{This figure is an example of a rigid variation applied to a function \(u[\bold{X}]\) with a jump in \(x_i\). The effect of the variation on the function is the same as moving the jump from \(x_i\) to \(x_i+h\).}
\end{figure}%
We now compute the first and second variations of the functional with respect to rigid variations. 
\begin{definition}[Rigid \((s,p)\)-Laplacian]
Let \(\bold{X}\in\X_T\) and \(u=u[\bold{X}]\), \(s\in (0,1)\) and \(p \geq 1\). We define the \emph{rigid s-fractional p-Laplacian} as the integral (in the sense of the principal value)
\begin{equation*} \label{def:laplaciano frazionario rigido}
\left(- \D \right)^s_p u(x) \coloneqq 2 \int_\R \frac{\left| u(x)-u(y)+1\right|^p-\left|u(x)-u(y)\right|^p}{\left|x-y\right|^{1+sp}}\, \ud y,
\end{equation*}
which is finite if and only if \(\bold{X}\in\X^0_T\).
\end{definition}

\begin{proposition}[First and second variations for \(\F^s_p\)] \label{prop: rigid variations}
Let \(s \in (0,1)\), \(p\geq 1\). For every \(\bold{X}\in\X_T^0\), and every \(i, j \in \{1,\ldots,T\}\), \(i \neq j\) it holds
\begin{align}
 \partial_{x_i} \F^s_p \left[\bold{X} \right] &= \left(- \D \right)^s_p u(x_i),  \label{eq: var prima s>0}\\
 \partial^2_{x_i x_j} \F^s_p \left[\bold{X} \right] &= 2 \sum_{k\in\Z} \frac{2|u(x_i)-u(x_j)|^p-|u(x_i)-u(x_j)+1|^p-|u(x_i)-u(x_j)-1|^p}{|x_i-x_j-kT|^{1+sp}}, \label{eq: var seconda mista s>0}\\
 \partial^2_{x_i^2}\F^s_p \left[\bold{X} \right] &= -2 \sum_{k\in\Z}\sum_{j\neq i} \frac{2|u(x_i)-u(x_j)|^p-|u(x_i)-u(x_j)+1|^p-|u(x_i)-u(x_j)-1|^p}{|x_i-x_j-kT|^{1+sp}}. \label{eq: var seconda pura s>0}  
\end{align}
\end{proposition}

\begin{proof}
Let us fix \(h_0>0\) so that \(\bold{X}\in\X_T^{h_0}\), and \(0<h<h_0\) (for symmetry reasons, it is enough to prove the result for \(h>0\)). We compute 
\begin{align*}
     \F^s_p\left[\bold{X}+he_i\right]-\F^s_p\left[\bold{X}\right]=& \int_0^T \int_\R \frac{\left\lvert u(x)-u(y)+\rchi_{I_h}(x)-\rchi_{\I_h}(y)\right\rvert^p - \left\lvert u(x)-u(y) \right\lvert^p}{|x-y|^{1+sp}}\, \ud y \ud x \\
     =& \int_{I_h} \int_{\R\backslash \I_h} \frac{\left\lvert u(x)-u(y)+1 \right\rvert^p -\left\lvert u(x)-u(y) \right\lvert^p}{|x-y|^{1+sp}} \, \ud y \ud x \\
     &+ \int_{[0,T]\backslash I_h} \int_{\I_h} \frac{\left\lvert u(x)-u(y)- 1 \right\rvert^p -\left\lvert u(x)-u(y) \right\lvert^p }{|x-y|^{1+sp}} \,\ud y \ud x\\
     =& 2 \int_{I_h} \int_{\R\backslash \I_h} \frac{\left\lvert u(x)-u(y)+1 \right\rvert^p -\left\lvert u(x)-u(y) \right\lvert^p}{|x-y|^{1+sp}} \,\ud y \ud x.
\end{align*}
where in the last equality we have used the periodicity of \(u\), the symmetry of the domains and Tonelli's Theorem.
Now, we want to take the limit sending \(h \to 0^+\) and in order to do that, we take \(\sigma>0\) such that \(\sigma> \frac{h}{2} \) and we define \(\I^*_h\coloneqq \I_h\backslash I_h\). Moreover, let us now take \(\xi_h=x_i+\frac{h}{2}\) so that \(I_h= B_{\frac{h}{2}}(\xi_h)\). Therefore we get 
\begin{align*}
	  \F^s_p\left[\bold{X}+he_i\right] -\F^s_p\left[\bold{X}\right] =& 2 \int_{I_h} \int_{B_\sigma^c(\xi_h)\backslash \I^*_h} \frac{\left\lvert u(x)-u(y)+1 \right\rvert^p -\left\lvert u(x)-u(y) \right\lvert^p}{|x-y|^{1+sp}} \, \ud y \ud x \\ 
    &+2 \int_{I_h} \int_{B_\sigma(\xi_h)\backslash I_h} \frac{\left\lvert u(x)-u(y)+1 \right\rvert^p -\left\lvert u(x)-u(y) \right\lvert^p}{|x-y|^{1+sp}}\, \ud y \ud x \\
    \eqqcolon& A_1^h + A_2^h.
\end{align*} 
We show immediately that \(A_2^h\) is identically zero:
\begin{align*}
    A_2^h =& 2\int_{I_h} \int_{B_\sigma(\xi_h)\backslash B_{\frac{h}{2}}(\xi_h)} \frac{\left\lvert u(x)-u(y)+1 \right\rvert^p -\left\lvert u(x)-u(y) \right\lvert^p}{|x-y|^{1+sp}}\, \ud y \ud x\\
    =&2 \int_{\xi_h-\frac{h}{2}}^{\xi_h+\frac{h}{2}} \int_{\xi_h-\sigma}^{\xi_h-\frac{h}{2}} \frac{\left\lvert x-y\right\rvert^p - \left\lvert x-y-1\right\rvert^p}{\left\lvert x-y \right\rvert^{1+sp}}\,\ud y \ud x
    +2 \int_{\xi_h-\frac{h}{2}}^{\xi_h+\frac{h}{2}} \int_{\xi_h+\frac{h}{2}}^{\xi_h+\sigma} \frac{\left\lvert x-y+1\right\rvert^p - \left\lvert x-y\right\rvert^p}{\left\lvert x-y \right\rvert^{1+sp}}\,\ud y \ud x\\
    =&2 \int_{-\frac{h}{2}}^{\frac{h}{2}} \int_{-\sigma}^{-\frac{h}{2}} \frac{\left\lvert x-y\right\rvert^p - \left\lvert x-y-1\right\rvert^p}{\left\lvert x-y \right\rvert^{1+sp}}\,\ud y \ud x
    +2 \int_{-\frac{h}{2}}^{\frac{h}{2}} \int_{\frac{h}{2}}^{\sigma} \frac{\left\lvert x-y+1\right\rvert^p - \left\lvert x-y\right\rvert^p}{\left\lvert x-y \right\rvert^{1+sp}}\,\ud y \ud x= 0,
 \end{align*}
 where we used the changes of variable \(x'=x-\xi_h\) and \(y'=y-\xi_h\) in the second equality, and the changes of variable \(x'=-x\) and \(y'=-y\) in the last equality. Moreover, in the first equality we have used that
 \begin{align*}
 &u(x)-u(y)= x-y-1 \qquad &&\text{ whenever } x\in I_h \text{ and } y \in\left(\xi_h-\sigma, \xi_h - \frac{h}{2}\right),\\
 &u(x)-u(y) =x-y \qquad &&\text{ whenever } x\in I_h \text{ and } y \in\left(\xi_h + \frac{h}{2}, \xi_h+\sigma\right).
 \end{align*}
Now, taking the limit as \(h\to0\), we compute
\begin{equation*}
        \lim_{\substack{h \to 0^+}} \frac{\F^s_p\left[\bold{X}+he_i\right]- \F^s_p\left[\bold{X}\right]}{h} = \lim_{\substack{\sigma \to 0^+}} 2 \int_{\R\backslash B_\sigma(x_i)} \frac{\left\lvert u(x_i)-u(y)+1 \right\rvert^p -\left\lvert u(x_i)-u(y) \right\lvert^p}{|x_i-y|^{1+sp}} \, \ud y,
\end{equation*}
where on \(\frac{1}{h}A_1^h\) we used the Mean Value Theorem as \(h \to 0^+\), arriving to \cref{eq: var prima s>0}.
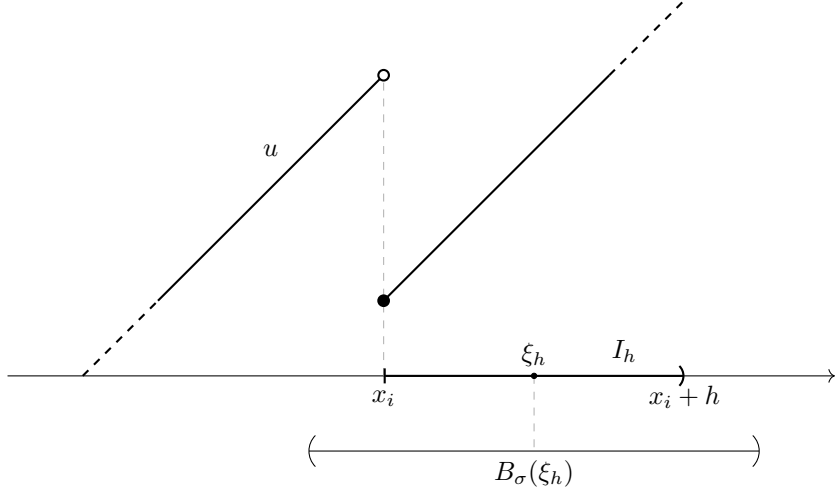
\begin{figure}
	\centering
	\begin{tikzpicture}
		\draw[->] (-5,0)--(6,0);
		
		\draw[dashed,black!30] (0,0) -- (0,4);
		
		\draw[dashed,black!30] (2,0) -- (2,-1);
		
		\draw[thick] (-3,1)--(0,4) node[pos=.5,above=.3cm]{\(u\)};
		\draw[dashed,thick] (-4,0)--(-3,1);
		\draw[fill=white,thick] (0,4) circle (2pt);
		
		\draw[thick]	(0,1)--(3,4);
		\draw[dashed,thick]	(3,4)--(4,5);
		\filldraw[thick] (0,1) circle (2pt);
		
		\draw[thick,arrows={Bar[width=.2cm]-Arc Barb[length=.2cm,arc=90]}] (0,0)node[below=.07cm]{\(x_i\)}--(4,0)node[below]{\(x_i+h\)} node[pos=.8,above=.05cm]{\(I_h\)}; 
		
		\filldraw (2,0) circle (1pt) node[above]{\(\xi_h\)};
		
		\draw[black, arrows =  {Arc Barb[length=.3cm,arc=90]-Arc Barb[length=.3cm,arc=90]}] (-1,-1) -- (5,-1) node[below, pos=.5]{\(B_{\sigma}(\xi_h)\)}; 
	\end{tikzpicture}
    \label{fig: insiemi conto variaizone prima} 
    \caption{The objects in the proof of \Cref{prop: rigid variations}.}
\end{figure}%

Now, we want to prove \cref{eq: var seconda mista s>0}. To do so, we consider a rigid variation in the direction of \(e_j\) where \(i \neq j\); that is, we take the interval \(J_h=[x_j,x_j+h)\), the set \(\J_h=\bigcup_{k\in\Z}(J_h+kT)\) and the function \(\rchi_{\J_h}\). We get
\begin{align*}
    \partial_{x_i}\F^s_p\left[\bold{X}+h e_j\right]=& 2\int_\R \frac{|u(x_i)-u(y)-\rchi_{\J_h}(y)+1|^p-|u(x_i)-u(y)-\rchi_{\J_h}(y)|^p}{|x_i-y|^{1+sp}}\, \ud y\\
    =& 2 \sum_{k\in\Z} \int_0^T \frac{|u(x_i)-u(y)-\rchi_{J_h}(y)+1|^p-|u(x_i)-u(y)-\rchi_{J_h}(y)|^p}{|x_i-y-kT|^{1+sp}}\, \ud y\\
    =&2 \sum_{k\in\Z} \int_{J_h} \frac{|u(x_i)-u(y)|^p-|u(x_i)-u(y)-1|^p}{|x_i-y-kT|^{1+sp}}\, \ud y\\
    &+ 2\sum_{k\in\Z} \int_{J_h^c\cap[0,T]} \frac{|u(x_i)-u(y)+1|^p-|u(x_i)-u(y)|^p}{|x_i-y-kT|^{1+sp}}\, \ud y.
\end{align*}
Writing now the difference quotient, we get
\begin{align*}
    &\frac{\partial_{x_i}\F^s_p\left[\bold{X}+h e_j\right]-\partial_{x_i}\F^s_p\left[\bold{X}\right]}{h} \\
   =&\frac{2}{h} \sum_{k\in\Z} \int_{J_h} \frac{2|u(x_i)-u(y)|^p-|u(x_i)-u(y)+1|^p-|u(x_i)-u(y)-1|^p}{|x_i-y-kT|^{1+sp}} \,\ud y.
\end{align*}
Taking the limit for \(h\to 0\) we finally get 
\[
\partial^2_{x_i x_j}\F^s_p\left[\bold{X}\right] =2 \sum_{k\in\Z} \frac{2|u(x_i)-u(x_j)|^p-|u(x_i)-u(x_j)+1|^p-|u(x_i)-u(x_j)-1|^p}{|x_i-x_j-kT|^{1+sp}}.
\]
As for the \cref{eq: var seconda pura s>0}, it is trivially implied by the following property
\begin{equation} \label{eq: somma zero hessiana}
\partial^2_{x_i^2}\F^s_p\left[\bold{X}\right] = -\sum_{j\neq i} \partial^2_{x_i x_j}\F^s_p\left[\bold{X}\right].
\end{equation}
\end{proof}

\begin{remark} \label{rmk: somma zero}
Taking the matrix \(F\left[\bold{X} \right] \coloneqq (F_{ij}\left[\bold{X}\right])_{i,j}\), where \(F_{ij} \left[ \bold{X} \right] \coloneqq \partial^2_{x_i x_j}\F^s_p\left[\bold{X}\right]\), we observe that from the property in \cref{eq: somma zero hessiana}, it follows that the sum of all the elements on a row or a column of this matrix is zero. 
\end{remark} 

Up to this point, we have only considered regular configurations; it is our final goal to show that critical points are always found in this class of configurations. To do so, we now show that it is always convenient, in order to lower the energy, to separate overlapping jumps.
\begin{lemma} \label{lemma: cuspidi}
	Let \(s\in(0,1)\) and \(p>1\) such that \(sp<1\). Let \(\bold{X}\in\X_T\) with \(1<m_i\coloneqq \#\{j\in\{1,\ldots,T\}\,|\, x_j=x_i\}\) the multiplicity of the jump in \(x_i\), for some \(i\in\{1,\ldots,T\}\). Then
	\[
	\partial^{\pm}_{x_i} \F^s_p[\bold{X}] =\mp \infty.
	\]
\end{lemma}

\begin{proof}
	Recalling the computations in the proof of \Cref{prop: rigid variations}, we have
	
\begin{align*}
	& \F^s_p\left[\bold{X}+he_i\right] -\F^s_p\left[\bold{X}\right] \\
	=& 2 \int_{I_h} \int_{B_\sigma^c(x)\backslash \I^*_h} \frac{\left\lvert u(x)-u(y)+1 \right\rvert^p -\left\lvert u(x)-u(y) \right\lvert^p}{|x-y|^{1+sp}} \,\ud y \ud x \\
	&+2 \int_{I_h} \int_{B_\sigma(x)\backslash I_h} \frac{\left\lvert u(x)-u(y)+1 \right\rvert^p -\left\lvert u(x)-u(y) \right\lvert^p}{|x-y|^{1+sp}} \,\ud y \ud x \\
	=& 2\int_{I_h} \int_{B_\sigma^c(x)\backslash \I^*_h} \frac{\left\lvert u(x)-u(y)+1 \right\rvert^p -\left\lvert u(x)-u(y) \right\lvert^p}{|x-y|^{1+sp}} \,\ud y \ud x  \\
	&+2 \int_{I_h} \int_{B_{\sigma}(\xi_h)\setminus B_{\frac{h}{2}}(\xi_h)}  \frac{\left\lvert u(x)-u(y)+1 \right\rvert^p -\left\lvert u(x)-u(y) \right\lvert^p}{|x-y|^{1+sp}} \,\ud y \ud x \\
	=& A_2^h + O(h).
\end{align*} 
Computing \(A_2^h\) we get to 
\begin{align*}
		A_2^h=& 2\int_{I_h}\int_{\xi_h- \sigma}^{\xi_h-\frac{h}{2}} \frac{|x-y+1-m_i|^p-|x-y-m_i|^p}{|x-y|^{1+sp}}\, \ud y \ud x \\
		&+ 2\int_{I_h}\int_{\xi_h+\frac{h}{2}}^{\xi_h+\sigma} \frac{|x-y+1|^p-|x-y|^p}{|x-y|^{1+sp}}\, \ud y \ud x \\
        \eqqcolon& B^h+C^h.
\end{align*}
We observe that, given the integration intervals in \(B_h\), we have that \(x-y>0\), and using the Taylor expansion of the functions \(f(t)=(m_i-1-t)^p\) and \(g(t)=(m_i-t)^p\) for \(|t|<m_i-1\), we easily get
\begin{align*}
B^h=&2\int_{I_h}\int_{\xi_h- \sigma}^{\xi_h-\frac{h}{2}} \frac{(m_i-1-(x-y))^p-(m_i-(x-y))^p}{(x-y)^{1+sp}}\, \ud y \ud x \\
=& 2\int_{I_h}\int_{\xi_h-\sigma}^{\xi_h-\frac{h}{2}} \frac{\left((m_i-1)^p-m_i^p\right)+O(x-y)}{(x-y)^{1+sp}} \ud y \ud x \\
=&\left((m_i-1)^p-m_i^p\right)\frac{2h^{1-sp}}{sp(1-sp)} + O(h^{2-sp}).
\end{align*}
For the second integral we observe that \(y-x>0\) and we use the same Taylor expansion to get
\begin{align*}
	C^h=&2\int_{I_h}\int_{\xi_h+\frac{h}{2}}^{\xi_h+ \sigma} \frac{(1-(y-x))^p-(y-x)^p}{(y-x)^{1+sp}}\, \ud y \ud x \\
	=& 2\int_{I_h}\int_{\xi_h+\frac{h}{2}}^{\xi_h+ \sigma} \frac{1+O(y-x)}{(y-x)^{1+sp}}\, \ud y \ud x \\
	=&  \frac{2h^{1-sp}}{sp(1-sp)} +O(h^{2-sp}).
\end{align*} 
Combining the two integrals we get 
\[
	A_2^h= \left((m_i-1)^p-m_i^p+1\right)\frac{2h^{1-sp}}{sp(1-sp)} +O(h^{2-sp}).
\]
Using the inequality \((a-1)^p-(a^p-1)<0\), true for \(p,a>1\), we conclude the proof.
	
\end{proof}

Observe that this proof does not work for \(p=1\). Indeed, the first variation of the energy is finite for \(h\to0\), because the coefficient of the leading-order term identically vanishes. In the next Lemma we will show that the energy is lowered when you separate a jump of a small enough quantity.

\begin{lemma} \label{lemma: cuspidi p=1}
	Let \(s\in(0,1)\), \(p=1\) and let \(\bold{X}\in\X_T\) with multiplicity \(m_i>1\) in \(x_i\), for some \(i\in\{1,\ldots,T\}\). Then there exists \(\bold X'\in\X_T\) with multiplicity \(m_i-1\) in \(x_i\) such that 
	\[
	\F^s_1[\bold X] > \F^s_1[\bold X'].
	\]
\end{lemma}
\begin{proof}
	Without loss of generality let us take \(x_i=0\) and \(u[\bold X](0)=0\). With a slight abuse of notation, denote by \(\bold X^+ \coloneqq \bold X + h e_i\) and \(\bold X^- \coloneqq \bold X - h e_i\), for some small \(h>0\), and let \(u^\pm=u[\bold X^\pm]\).
	
	We start by analysing the right variation \(\bold{X}^+\). Let \(I_h^+ \coloneqq (0,h)\) and \(\I_h^+\) be the union of its \(T\)-translated intervals. Taking the difference of the energies, manipulating the domains using periodicity, and applying the change of variables \(y=t+x\), we get
	\begin{align*}
		\Delta^+(h) &\coloneqq \F^s_1[\bold X^+] - \F^s_1[\bold X] = \int_0^T \int_\R \frac{|u^+(x)-u^+(y)|-|u(x)-u(y)|}{|x-y|^{1+s}} \, \ud y \ud x \\
		&= 2\int_0^h \int_{\R\setminus \I_h^+} \frac{|u(x)-u(y)+1|-|u(x)-u(y)|}{|x-y|^{1+s}} \, \ud y \ud x \\
		&= 2\int_0^h \int_{\R\setminus (\I_h^+-x)} \frac{|u(x)-u(t+x)+1|-|u(x)-u(t+x)|}{|t|^{1+s}} \, \ud t \ud x \\
		&= 2\int_0^h \int_{\R\setminus (I_h^+-x)} \frac{|u(x)-u(t+x)+1|-|u(x)-u(t+x)|}{|t|^{1+s}} \, \ud t \ud x + O(h^2).
	\end{align*} 
	In the last equality, we ignored the intervals \((kT,kT+h)\) for \(k\neq 0\); this produces an error of order \(O(h^2)\), which will be absorbed into the \(O(h^{2-s})\) term that will appear later.
	
	Now we define the quantity 
	\[
	F^+ \coloneqq 2\,\mathrm{P.V.}\int_\R \frac{|1-u(t)|-|u(t)|}{|t|^{1+s}} \, \ud t.
	\]
	Thanks to the symmetry of the function inside the integral around 0, \(F^+\) is finite. Furthermore, for any \(x\in(0,h)\), we can split this integral as:
	\[
	F^+ = 2\int_{\R \setminus (I_h^+-x)} \frac{|1-u(t)|-|u(t)|}{|t|^{1+s}} \, \ud t + 2\, \mathrm{P.V.} \int_{I_h^+-x}\frac{|1-u(t)|-|u(t)|}{|t|^{1+s}} \, \ud t.
	\]
	We now define the integrand:
	\[
	D^+(x,t) \coloneqq \frac{|u(x)-u(t+x)+1|-|u(x)-u(t+x)|}{|t|^{1+s}} - \frac{|1-u(t)|-|u(t)|}{|t|^{1+s}}.
	\]
	Summing and subtracting \(hF^+\) from \(\Delta^+(h)\), we can rewrite the energy difference as:
	\begin{align}
		\Delta^+(h) &= hF^+ + \left(\Delta^+(h) - hF^+\right) \nonumber\\
		&= hF^+ + 2\int_0^h \left(\int_{\R\setminus (I_h^+-x)} D^+(x,t) \, \ud t - \mathrm{P.V.} \int_{I_h^+-x}\frac{|1-u(t)|-|u(t)|}{|t|^{1+s}} \, \ud t\right) \ud x \label{eq: cusp1 Detlta+} \\
       &  \hspace{11cm}+ O(h^2). \nonumber
	\end{align}
	We now focus on \(D^+(x,t)\). Since \(\bold X\) is a finite set of points in \([0,T)\), we can consider the minimal non-zero distance \(0<\delta \coloneqq \min\{|x_i-x_j|\,|\, j\in\{1,\dots,T\} \text{ and } x_i \neq x_j\}\). We can easily see that \(D^+(x,t)\equiv 0\) for \(|t|<\delta\). Furthermore, an explicit integration yields:
	\begin{equation} \label{eq: cusp1 conto 1}
	    \mathrm{P.V.} \int_{I_h^+-x}\frac{|1-u(t)|-|u(t)|}{|t|^{1+s}} \, \ud t = \frac{x^{-s}}{s}-\frac{(h-x)^{-s}}{s}-\frac{2(h-x)^{1-s}}{1-s}.
	\end{equation}
	Substituting \cref{eq: cusp1 conto 1} into \cref{eq: cusp1 Detlta+}, we get 
	\[
	\Delta^+(h) = hF^+ + 2\int_0^h \left(\int_{|t|>\delta} D^+(x,t) \, \ud t - \left(\frac{x^{-s}}{s}-\frac{(h-x)^{-s}}{s}-\frac{2(h-x)^{1-s}}{1-s}\right)\right) \ud x + O(h^2). 
	\]
	By symmetry, we can see that 
	\[
	2\int_0^h \left(\frac{x^{-s}}{s}-\frac{(h-x)^{-s}}{s}\right) \ud x = 0,
	\]
	and integrating the remaining term gives
	\[
	2\int_0^h \frac{2(h-x)^{1-s}}{1-s} \, \ud x = \frac{4}{(1-s)(2-s)}h^{2-s}=O(h^{2-s}). 
	\]
	Our next claim is that 
	\[
	\int_0^h\int_{|t|>\delta} D^+(x,t) \, \ud t \, \ud x = O(h^2).
	\]
	Indeed, \(D^+(x,t)\equiv 0\) if there is no jump point of \(u\) in \([t,t+x]\). If, on the other hand, there exists \(x_j\in[t,t+x]\), i.e. \(t\in[x_j-x,x_j]\), then \(D^+(x,t)\) is non-zero. Therefore, we can focus our attention on the set 
	\[
	X_\delta(x) \coloneqq \bigcup_{j:|x_j|>\delta} [x_j-x,x_j].
	\]
	This is a union of (not necessarily disjoint) intervals. Taking an upper bound for the numerator and the smallest possible denominator on these intervals, we get
	\begin{align*}
		\left\lvert\int_{|t|>\delta} D^+(x,t) \, \ud t \right\rvert &= \left\lvert\int_{X_\delta(x)} D^+(x,t) \, \ud t \right\rvert \\
		&\leq \sum_{|x_j|>\delta} \int_{x_j-x}^{x_j}\frac{2T+1}{|t|^{1+s}} \, \ud t \\
		&\leq (2T+1)\sum_{|x_j|>\delta} \frac{x}{|x_j|^{1+s}} \\
		&\leq 2T(2T+1)x\sum_{k\in\N} \frac{1}{k^{1+s}} \leq Cx.
	\end{align*}
	Computing the integral in \(x\), we get the claim. To summarize the results obtained so far:
	\begin{equation} \label{eq: cusp1 risul 1}
	    \Delta^+(h) = hF^+ + O(h^{2-s}).
	\end{equation}
	
	We now turn to the left variation, taking \(I_h^- \coloneqq (-h,0)\) and defining:
	\begin{align}
		\Delta^-(h) &\coloneqq \F^s_1[\bold X^-] - \F^s_1[\bold X] = \int_0^T \int_\R \frac{|u^-(x)-u^-(y)|-|u(x)-u(y)|}{|x-y|^{1+s}} \, \ud y \ud x \nonumber \\
		&= 2\int_{-h}^0 \int_{\R\setminus \I_h^-} \frac{|u(x)-u(y)-1|-|u(x)-u(y)|}{|x-y|^{1+s}} \, \ud y \ud x \nonumber\\
		&= 2\int_{-h}^0 \int_{\R\setminus (\I_h^--x)} \frac{|u(x)-u(t+x)-1|-|u(x)-u(t+x)|}{|t|^{1+s}} \, \ud t \ud x \nonumber\\
		&= 2\int_{-h}^0 \int_{\R\setminus (I_h^--x)} \frac{|u(x)-u(t+x)-1|-|u(x)-u(t+x)|}{|t|^{1+s}} \, \ud t \ud x + O(h^2), \label{eq: cusp1 Delta-}
	\end{align} 
    where in the second to last equality we have used the change of variable \(y\mapsto x+t\). Similarly, we define 
	\[
	F^- \coloneqq 2\, \mathrm{P.V.}\int_\R \frac{|m_i-1-u(t)|-|m_i-u(t)|}{|t|^{1+s}} \, \ud t.
	\]
	For any \(x\in(-h,0)\), we can split \(F^-\) in:
    \begin{equation}
    \begin{aligned} \label{eq: cusp1 F-}
        F^- =& 2\int_{\R \setminus (I_h^--x)} \frac{|m_i-1-u(t)|-|m_i-u(t)|}{|t|^{1+s}} \, \ud t \\
        &+ 2\, \mathrm{P.V.} \int_{I_h^--x}\frac{|m_i-1-u(t)|-|m_i-u(t)|}{|t|^{1+s}} \, \ud t. 
    \end{aligned}
    \end{equation}
	Proceeding as before, we set:
	\[
	D^-(x,t) \coloneqq \frac{|u(x)-u(t+x)-1|-|u(x)-u(t+x)|}{|t|^{1+s}} - \frac{|m_i-1-u(t)|-|m_i-u(t)|}{|t|^{1+s}}.
	\]
	Combining \cref{eq: cusp1 Delta-,eq: cusp1 F-}, we get:
	\begin{align*}
	\Delta^-(h) =& hF^- + 2\int_{-h}^0 \left(\int_{\R\setminus (I_h^--x)}D^-(x,t)\,\ud t - \mathrm{P.V.} \int_{I_h^--x}\frac{|m_i-1-u(t)|-|m_i-u(t)|}{|t|^{1+s}} \, \ud t\right) \ud x \\
    & \hspace{12cm}+ O(h^2).
	\end{align*}
	Using analogous computations to those for the right variation, we arrive at:
	\begin{equation} \label{eq: cusp1 risul 2}
	    \Delta^-(h) = hF^- + O(h^{2-s}).
	\end{equation}
	Let us now suppose, by contradiction, that the original configuration \(\bold X\) is a local minimiser. This would imply that \(\Delta^-(h)>0\) and \(\Delta^+(h)>0\). Since \cref{eq: cusp1 risul 1,eq: cusp1 risul 2} are dominated by their linear terms for \(h\) small enough, this would imply \(F^+>0\) and \(F^->0\). We will show that \(F^++F^- < 0\), which gives a contradiction. Indeed, considering the function \(g(z) \coloneqq |z-1|-|z|\), we can write the sum as:
	\[
	F^++F^- = 2 \, \mathrm{P.V.}\int_\R \frac{g(u(y))-g(u(y)-m_i+1)}{|y|^{1+s}} \, \ud y. 
	\]
	Since \(m_i > 1\), it holds that \(u(y) > u(y)-m_i+1\). Furthermore, \(g\) is a monotonically non-increasing function and strictly decreasing on \((0,1)\), meaning \(g(u(y)) \leq g(u(y)-m_i+1)\) everywhere, and strictly less on a set of positive measure. Therefore, the integral is strictly negative, which concludes the proof.
	
\end{proof}
Given the two Lemmas above, we can assume henceforth every configuration to be regular.

\begin{proposition} \label{prop: crit per s>0}
Let \(p\geq 1\), \(s\in(0,1)\) and \(sp<1\). The configuration \(\bold{X}^{\mathrm{eq}}\coloneqq(x_i=i)_{i \in \Z}\) is a critical point for the functional \(\F^s_p\) and \(F[\bold{X}]\) is positive definite for every \(\bold{X}\in\X^0_T\).
\end{proposition}
\begin{proof}
We start by showing that the first variation, when computed on the equispaced configuration \(\bold{X}^\mathrm{eq}\), is zero. Indeed we have
\begin{align*}
    \partial_{x_i}\F^s_p\left[\bold{X}^{\mathrm{eq}}\right] = \sum_{l\in\Z} 2 \int_0^1 \frac{|y-1|^p-|y|^p}{|y-l|^{1+sp}} \ud y=\sum_{l\in\Z} f_l,
\end{align*}
and now, taking respectively the term corresponding to \(l\) and \(-l+1\), and observe that for the latter
\begin{align*}
    f_{-l+1}=&2\int_0^1 \frac{|y-1|^p-|y|^p}{|y+l-1|^{1+sp}} \ud y\\
    =& 2 \int_{-1}^0 \frac{|y|^p-|y+1|^p}{|y+l|^{1+sp}} \ud y \\
    =& 2 \int_{0}^{1} \frac{|y|^p-|y-1|^p}{|y-l|^{1+sp}} \ud y \\
    =& -2\int_0^1 \frac{|y-1|^p-|y|^p}{|y-l|^{1+sp}} \ud y=-f_l,
\end{align*}
where in the second equality we used the change of variable \(y-1 \mapsto y\). Therefore, all the terms in the first variation cancel each other. 

As for the second variation, we look at the matrix  \(F[\bold{X}]\) (see \cref{rmk: somma zero}) and we start by observing that all the mixed terms \(F_{ij}[\bold{X}]\), when \(i \neq j\), are negative, thanks to the inequality:
\begin{equation} \label{eq: convessità al contrario}
    2|t|^p-|t-1|^p-|t+1|^p < 0, \qquad \forall p> 1, \forall t\in\R.
\end{equation}
Hence, we show that the matrix \(F[\bold{X}]\) is positive-definite for \(p>1\). To relieve the notation, we will write \(F_{ij}=F_{ij}[\bold{X}]\). Take \(\xi\in\R^T\), we start by observing that we can decompose the sum \(F\xi\cdot\xi\) and rearrange the terms appropriately to get
\begin{align} \label{eq: conto positività var seconda}
    F\xi\cdot\xi= &\sum_{i}F_{ii} \xi_i^2 + \sum_{j\neq i} F_{ij}\xi_i\xi_j \\
    =& -\sum_{i=1}^T\sum_{j\neq i}F_{ij}\xi_i^2 + \sum_{j\neq i} F_{ij}\xi_i\xi_j \nonumber\\
    =& -\sum_{i=1}^T \sum_{j>i} F_{ij} \left(\xi_i-\xi_j\right)^2,\nonumber
\end{align}
where in the second equality we have used the property stated in \cref{eq: somma zero hessiana}.
As we showed in \cref{eq: convessità al contrario}, the elements \(F_{ij},\,i\neq j\), are negative; therefore, the matrix is positive definite.

For \(p=1\) the strict inequality in \cref{eq: convessità al contrario} holds only if \(|t|<1\); therefore, a priori, the matrix \(F\) is only positive-semidefinite. We will now prove that, for every \(i=1,\ldots,T\) there exists at least one \(j\neq 1\) such that \(F_{ij}<0\) and, in doing so,  that the matrix \(F\) is indeed positive-definite for \(p=1\) as well.

Without loss of generality we can fix \(i=1\) and \(x_1=0\). The condition for \(|t|<1\) in \cref{eq: convessità al contrario} directly translate to the condition \(|u(x_1)-u(x_j)|<1\) for every \(j \in \{1,\ldots,T\}\) in \cref{eq: var seconda pura s>0}. If, by contradiction, for every \(j>1\) holds \(|u(x_1)-u(x_j)|\geq1\), since the jumps are separated and quantised to be 1, the only possibility is for \(u(x_j)-u(x_1)\geq1\). Then it would also be true that \(u(x_T)-u(x_1)\geq 1\) and therefore \(u(T)-u(0)\geq u(x_T)-u(x_1)\geq 1\), but this is in contradiction with the periodicity of the function \(u\) and the presence of \(T\) jumps in \([0,T)\).

\end{proof}

\begin{theorem}[Equispaced minimiser] \label{thm: min s>0}
Let \(p\geq 1\), \(s\in(0,1)\) and \(sp<1\). The unique critical point (up to translations) of the energy \(\F^s_p\) in \(\X^0_T\) is \(\bold{X}^\mathrm{eq}\). Moreover, \(\bold{X}^\mathrm{eq}\) is the unique (up to translations) local and global minimiser in \(\X_T\).
\end{theorem}
\begin{proof}
Let \(\bold X\) be a local minimiser of \(\F^s_p\) in \(\X_T\); by \Cref{lemma: cuspidi} it naturally follows that \(\bold X \in \X^0_T\), so that it is a critical point for the energy. By \Cref{prop: crit per s>0}, the second variation of the energy is strictly positive; this fact, together with the convexity of the space \(\X^0_T\), implies that there is at most one critical point of the energy. As for \Cref{prop: crit per s>0}, the first variation vanishes on the equispaced configuration; therefore, following the previous considerations, it is the only local (and global) minimiser.
    
\end{proof}

\subsection{The critical and super-critical case: \(sp \geq 1\)}
In this setting the energy of a configuration \(\bold{X}\) is infinite, indeed \(u[\bold{X}]\notin W^{s,p}(\mathbb{T}^1_T)\) for \(sp\geq 1\). To circumvent this issue, we adopt a standard regularisation approach. Let \(\rho\) be a standard mollifier supported in \((-1,1)\) and for any \(\eps>0\), let \(\rho_\eps(\cdot)=\frac{1}{\eps}\rho(\frac{\cdot}{\eps})\). Now, let \(\bold{X}\in\mathbb{X}_T\), \(u=u[\bold{X}]\) and \(u^\eps\coloneqq u \ast \rho_\eps\). For every \(\bold X \in \X_T\) we define 
\[
\F^{s,\eps}_p[\bold{X}] \coloneqq \F^{s}_p(u^\varepsilon[\bold{X}]).
\]
From now on, we will let \(\eps>0\) be fixed. 
\begin{remark}\label{rmk: variazione rigida moscia}
    It is clear that,  by the linearity of the convolution, doing a rigid variation of \(u^\eps\) in \(x_i \in \bold{X}\) is equivalent to compute a rigid variation of the function \(u\) in \(x_i\) and then apply the convolution; indeed we have
    \[
        u^\eps[\bold{X}+he_i](x)= \left( u+ \rchi_{\I_h}\right)\ast \rho_\eps(x).
    \]
\end{remark}
\begin{figure}[h]
	\centering
	\begin{tikzpicture}
		
		\draw[dashed,black!30] (0,0) node[black,below=.08cm]{\(x_i\)} -- (0,4);
		\draw[dashed,black!30] (-.5,0) node[black, above left]{\(x_i\!-\!\varepsilon\)} -- (-.5,3.5);
		\draw[dashed,black!30] (1,0) node[black,below]{\(x_i\!+\!h\)} -- (1,3.5);
		\draw[dashed,black!30] (1.5,0) node[black, above right]{\(x_i\!+\!h\!+\!\varepsilon\)} -- (1.5,2.5);
		
		\draw[black!30,thick] (-.5,3.5) to[out=45,in=95, distance=15] (0,2.5);
		\draw[black,thick] (.5,4.5) to[out=45,in=95, distance=15] (1,3.5);
		\draw[thick] (-3,1)--(.5,4.5) node[pos=.4,above=.3cm]{\(u^\varepsilon\)};
		\draw[dashed,thick] (-4,0)--(-3,1);
		
		\draw[thick]	(1,2)--(3,4);
		\draw[dashed,thick]	(3,4)--(4,5);
		\draw[black!30,thick] (0,2.5) to[out=-85,in=225, distance=15] (.5,1.5);
		\draw[black!30,thick] (.5,1.5) -- (1.5,2.5);
		\draw[black,thick] (1,3.5) to[out=-85,in=225, distance=15] (1.5,2.5);
		
		\draw[arrows={-Stealth}] (.5,1.7) -- (.5,4.3);
		
		\draw[->] (-5,0)--(5,0);
	\end{tikzpicture}
	\label{fig: schema variazione rigida moscia}
	\caption{This figure is an example of a rigid variation applied to a regularised function \(u^\eps[\bold{X}]\) with a jump in \(x_i\). The effect of the variation on the function is the same as moving the regularised transition from \(x_i\) to \(x_i+h\).}
\end{figure}%

To relieve the notation we will denote by \(\rho_\eps^i(x)\coloneqq \rho_\eps(x-x_i)\) and with \(\tilde \rho_\eps^i(x) \coloneqq \sum_{k\in\Z}\rho_\eps(x-x_i-kT)\).

\begin{proposition}[First and second variations for \(\F^{s,\eps}_p\)] \label{prop: rigid variations super-critical}
Let  \(p> 1\), \( \frac{1}{p} \leq s < 1\), \( \bold{X}\in\X^0_T\), \(u^\varepsilon=u^\varepsilon[\bold{X}]\) and let us take \(i, j \in \{1,\ldots,T\}\) with \(i \neq j\). Then it holds true
\begin{align}
 \partial_{x_i} \F^{s,\eps}_p \left[\bold{X} \right] &=  2\int_{B_\varepsilon(x_i)}\left(\left(- \Delta \right)^s_p u^\eps\ast \rho_\eps\right)(x)\,  \ud x, \label{eq: var prima supercritica s>0}\\ 
 \strut
 \partial^2_{x_i x_j} \F^{s,\eps}_p \left[\bold{X} \right] &=p(p-1) \int_{\R} \int_{\R} \frac{\left\lvert u^\eps(x)-u^\eps(y)\right\rvert^{p-2} }{|x-y|^{1+sp}} \left(\tilde \rho_\eps^j(x)- \tilde \rho_\eps^j(y)\right) \left(\rho_\eps^i(x) -\rho_\eps^i(y) \right)\,\ud y \ud x, \label{eq: var seconda mista problematica}\\
 \strut
 \partial^2_{x_i^2} \F^{s,\eps}_p \left[\bold{X} \right] &=-p(p-1) \sum_{j\neq i}\int_{\R} \int_{\R} \frac{\left\lvert u^\eps(x)-u^\eps(y)\right\rvert^{p-2} }{|x-y|^{1+sp}} \left(\tilde \rho_\eps^j(x)- \tilde \rho_\eps^j(y)\right) \left(\rho_\eps^i(x) -\rho_\eps^i(y) \right)\,\ud y \ud x. \label{eq: var seconda pura problematica}
\end{align}
Moreover, if \(\bold X \in \X_T^{4\eps}\) then
\begin{align}
 \partial^2_{x_i x_j} \F^{s,\eps}_p \left[\bold{X} \right] &= -2p(p-1) \sum_{k \in \Z} \int_{B_\eps(x_i)} \int_{B_\eps(x_j)} \frac{\left\lvert u^\eps(x)-u^\eps(y)\right\rvert^{p-2}}{|x-y-kT|^{1+sp}} \rho_\eps^j(y) \rho_\eps^i(x)\, \ud y \ud x,\label{eq: var seconda mista supercritica s>0}\\
 \partial^2_{x^2_i} \F^{s,\eps}_p \left[\bold{X} \right] &=  2p(p-1)  \sum_{j \neq i} \sum_{k \in \Z}\int_{B_\eps(x_i)} \int_{B_\eps(x_j)} \frac{\left\lvert u^\eps(x)-u^\eps(y)\right\rvert^{p-2}}{|x-y-kT|^{1+sp}} \rho_\eps^j(y)\rho_\eps^i(x)\, \ud y \ud x.
 \label{eq: var seconda pura supercritica s>0}
\end{align}
\end{proposition}
\begin{proof} We proceed by computing the difference quotient. Recall that \(I_h=[x_i,x_i+h)\) and \(\I_h\) is the union of the \(T\)-translated intervals. Then
    \begin{align*}
     & \frac{\F^{s,\eps}_p\left[\bold{X}+he_i\right]-\F^{s,\eps}_p\left[\bold{X}\right]}{h} \\
     =& \frac{1}{h} \int_0^T \int_{\R} \frac{\left\lvert u^\eps(x)-u^\eps(y)+\rchi^\eps_{I_h}(x)-\rchi^\eps_{\I_h}(y) \right\rvert^p -\left\lvert u^\eps(x)-u^\eps(y) \right\lvert^p}{|x-y|^{1+sp}}\, \ud y \ud x \\
     =& \frac{1}{h} \int_0^T \int_{\R} \frac{\left\lvert u^\eps(x)-u^\eps(y)+ \int_{I_h} \rho_\eps (x-z)\ud z-\int_{\I_h} \rho_\eps (y-z)\ud z\right\rvert^p -\left\lvert u^\eps(x)-u^\eps(y) \right\lvert^p }{|x-y|^{1+sp}}\, \ud y \ud x\\
      =& p \int_0^T \int_{\R} \frac{\left\lvert g^h_\eps(x,y)\right\rvert^{p-2} g^h_\eps(x,y) }{|x-y|^{1+sp}} \frac{\left(\int_{I_h} \rho_\eps (x-z)\ud z-\int_{\I_h} \rho_\eps (y-z)\ud z\right)}{h}\,\ud y \ud x.\\
\end{align*}
In the last equality, we have used the Mean Value Theorem with the function 
\[
f(t)=\left\lvert u^\eps(x)-u^\eps(y)+t(\rchi^\eps_{I_h}(x)-\rchi^\eps_{\I_h}(y))\right\rvert^p
\] 
where \(g_\eps^h\) is the function that satisfies 
\[
f(1)-f(0) = p|g_\eps^h(x,y)|^{p-1}g_\eps^h(x,y)(\rchi^\eps_{I_h}(x)-\rchi^\eps_{\I_h}(y)).
\]
We observe that \(g^h_\eps\) is a smooth function, converging uniformly to \(u^\eps(x)-u^\eps(y)\) as \(h\to0\), such that for every \((x,y) \in [0,T]\times \R\) it holds
\[
u^\eps(x)-u^\eps(y)-\left|\rchi^\eps_{I_h}(x)-\rchi^\eps_{\I_h}(y)\right|\leq g^h_\eps(x,y) \leq u^\eps(x)-u^\eps(y)+\left|\rchi^\eps_{I_h}(x)-\rchi^\eps_{\I_h}(y)\right|.
\]
Therefore, we can take the limit as \(h \to 0\), getting

\begin{align*}
    &\partial_{x_i} \F^{s,\eps}_p [\textbf{X}]\\
     =& p \int_0^T \int_{\R} \frac{\left\lvert u^\eps(x)-u^\eps(y)\right\rvert^{p-2} (u^\eps(x)-u^\eps(y)) }{|x-y|^{1+sp}} \left(\rho_\eps(x-x_i)-\sum_{k \in \Z} \rho_\eps(y-x_i-kT)\right) \ud y \ud x\\
    =& p\int_{B_\eps(x_i)} \int_{\R} \frac{\left\lvert u^\eps(x)-u^\eps(y)\right\rvert^{p-2} \left( u^\eps(x)-u^\eps(y)\right) }{|x-y|^{1+sp}} \rho_\eps(x-x_i) \,\ud y \ud x\\
    &- p \sum_{k \in \Z}
    \int_0^T \int_{B_\eps(x_i+kT)} \frac{\left\lvert u^\eps(x)-u^\eps(y)\right\rvert^{p-2} \left( u^\eps(x)-u^\eps(y)\right)} {|x-y|^{1+sp}} \rho_\eps(y-x_i-kT)\, \ud y \ud x\\
    =&p \sum_{k \in \Z} \int_{B_\eps(x_i)} \int_0^T \frac{\left\lvert u^\eps(x)-u^\eps(y)\right\rvert^{p-2} \left( u^\eps(x)-u^\eps(y)\right) }{|x-y-kT|^{1+sp}} \rho_\eps(x-x_i)\, \ud y \ud x\\
    &- p\sum_{k \in \Z}
    \int_0^T \int_{B_\eps(x_i)} \frac{\left\lvert u^\eps(x)-u^\eps(y)\right\rvert^{p-2} \left(u^\eps(x)-u^\eps(y)\right)} {|x-y-kT|^{1+sp}} \rho_\eps(y-x_i)\, \ud y \ud x\\
    =&p \sum_{k \in \Z} \int_{B_\eps(x_i)} \int_0^T \frac{\left\lvert u^\eps(x)-u^\eps(y)\right\rvert^{p-2} \left( u^\eps(x)-u^\eps(y)\right) }{|x-y-kT|^{1+sp}} \rho_\eps(x-x_i)\, \ud y \ud x\\
    &- p\sum_{k \in \Z}
     \int_{B_\eps(x_i)} \int_0^T
     \frac{\left\lvert u^\eps(x)-u^\eps(y)\right\rvert^{p-2} \left(u^\eps(x)-u^\eps(y)\right)} {|x-y-kT|^{1+sp}} \rho_\eps(y-x_i) \,\ud x \ud y\\
      =& 2p \int_{B_\eps(x_i)} \int_\R \frac{\left\lvert u^\eps(x)-u^\eps(y)\right\rvert^{p-2} \left( u^\eps(x)-u^\eps(y)\right) }{|x-y|^{1+sp}} \rho_\eps(x-x_i) \,\ud y \ud x.\\
\end{align*}
       Now, we are ready to prove \cref{eq: var seconda mista supercritica s>0}. Let \(j \neq i\), then
\begin{align*}
    &\frac{\partial_{x_i}\F^{s,\eps}_p \left[\bold{X}+he_j\right]-\partial_{x_i}\F^{s,\eps}_p\left[\bold{X}\right]}{h} \\
    =& \frac{2p}{h} \int_{B_\eps(x_i)} \int_{\R} \frac{\left\lvert u^\eps(x)-u^\eps(y)+ \rchi^\eps_{J_h}(x)-\rchi^\eps_{\J_h}(y) \right\rvert^{p-2} (u^\eps(x)-u^\eps(y)+ \rchi^\eps_{J_h}(x)-\rchi^\eps_{\J_h}(y))}{|x-y|^{1+sp}} \rho_\eps(x-x_i)\, \ud y \ud x\\
     &-\frac{2p}{h} \int_{B_\eps(x_i)} \int_{\R} \frac{ \left\lvert u^\eps(x)-u^\eps(y) \right\lvert^{p-2}(u^\eps(x)-u^\eps(y))}{|x-y|^{1+sp}}\rho_\eps(x-x_i) \,\ud y \ud x \\
    =& \frac{2p(p-1)}{h} \int_{B_\eps(x_i)} \int_{\R} \frac{\left\lvert g_\eps^h(x,y)\right\rvert^{p-2}}{|x-y|^{1+sp}} \frac{\left( \int_{J_h} \rho_\eps (x-z)\ud z-\int_{\J_h} \rho_\eps (y-z)\ud z \right)}{h} \rho_\eps(x-x_i)\,\ud y \ud x,\\
\end{align*}
using, in the last equality, the Mean Value Theorem as in the first argument.
Then, we can take the limit as \(h \to 0\) to get
\begin{equation*} 
    \partial^2_{x_i x_j} \F^{s,\eps}_p \left[\bold{X} \right] = 2p(p-1) \int_{B_\eps(x_i)} \int_{\R} \frac{\left\lvert u^\eps(x)-u^\eps(y)\right\rvert^{p-2}}{|x-y|^{1+sp}} \left(\rho_\eps^j(x)-\tilde \rho_\eps^j(y)\right) \rho_\eps^i(x) \,\ud y \ud x.
\end{equation*}
Let us observe that for \(p\in(1,2)\) the numerator in the fraction has negative exponent. This poses a potential integrability issue: due to the periodicity of \(u^\eps\) there exist distant points (at distance \(kT\) for instance) mapped to the same value. This can also happen for two points (or subsets) of \([0,T]\). Nevertheless, since the function \(u^\eps\) is never constant we show that whenever this happens it resolves into a finite contribution in the second variation. We begin by observing that we can consider, without loss of generality, a set \(A\) in which \(u^\eps\) is monotone, and another set \(B=A+c\) such that \(u^\eps|_A=u^\eps|_B\).
\begin{align*}
    \int_A\int_B\frac{\left\lvert u^\eps(x)-u^\eps(y)\right\rvert^{p-2}}{|x-y|^{1+sp}} \ud y \ud x &= \int_A\int_A\frac{\left\lvert u^\eps(x)-u^\eps(y)\right\rvert^{p-2}}{|x-y+c|^{1+sp}} \ud y \ud x \\
    &= \int_A\int_A \frac{\left\lvert (u^\eps)'(\xi)\right\rvert^{p-2}|x-y|^{p-2}}{|x-y+c|^{1+sp}} \ud y \ud x <+\infty,
\end{align*}
where we used the change of variables \(y\mapsto y-c\) and the Mean Value Theorem. The integral is indeed finite, since \(p\in(1,2)\) implies \(p-2>-1\).

Now, if we suppose \(B_\eps(x_i) \cap B_\eps(x_j) = \varnothing\) we easily get that 
\begin{align*}
    &\partial^2_{x_i x_j} \F^{s,\eps}_p \left[\bold{X} \right]\\
    =& 2p(p-1) \int_{B_\eps(x_i)} \int_{\R} \frac{\left\lvert u^\eps(x)-u^\eps(y)\right\rvert^{p-2}}{|x-y|^{1+sp}} \rho_\eps(x-x_j) \rho_\eps(x-x_i)\,\ud y \ud x\\
    &- 2p(p-1) \int_{B_\eps(x_i)} \int_{\R} \frac{\left\lvert u^\eps(x)-u^\eps(y)\right\rvert^{p-2}}{|x-y|^{1+sp}}\sum_{k \in \Z} \rho_\eps(y-x_j-kT)\, \ud y \ud x\\
    =&-2p(p-1) \sum_{k \in \Z} \int_{B_\eps(x_i)} \int_{B_\eps(x_j+kT)} \frac{\left\lvert u^\eps(x)-u^\eps(y)\right\rvert^{p-2}}{|x-y|^{1+sp}} \rho_\eps(y-x_j-kT)\rho_\eps(x-x_i)\, \ud y \ud x\\
    =& -2p(p-1) \sum_{k \in \Z} \int_{B_\eps(x_i)} \int_{B_\eps(x_j)} \frac{\left\lvert u^\eps(x)-u^\eps(y)\right\rvert^{p-2}}{|x-y-kT|^{1+sp}} \rho_\eps(y-x_j)\rho_\eps(x-x_i)\, \ud y \ud x.\\
\end{align*}
This gives us \cref{eq: var seconda mista supercritica s>0}. On the other hand, \cref{eq: var seconda pura supercritica s>0} follows from the property stated in \cref{eq: somma zero hessiana}, which is true also in the super-critical case.

Otherwise, if we suppose  \(B_\eps(x_i) \cap B_\eps(x_j) \neq \varnothing\), we get
\begin{align*}
&\partial^2_{x_i x_j} \F^{s,\eps}p \left[\bold{X} \right]\\
=& 2p(p-1) \int{\R} \int_{\R} \frac{\left\lvert u^\eps(x)-u^\eps(y)\right\rvert^{p-2} }{|x-y|^{1+sp}} \left(\rho_\eps(x-x_j)-\sum_{k \in \Z} \rho_\eps(y-x_j-kT)\right) \rho_\eps(x-x_i) \,\ud y \ud x.\\
=&\frac{2p(p-1)}{2} \int_{\R} \int_{\R} \frac{\left\lvert u^\eps(x)-u^\eps(y)\right\rvert^{p-2} }{|x-y|^{1+sp}} \left(\sum_{k \in \Z} \left(\rho_\eps(x-x_j-kT)- \rho_\eps(y-x_j-kT)\right)\right) \rho_\eps(x-x_i) \,\ud y \ud x\\
&-\frac{2p(p-1)}{2} \int_{\R} \int_{\R} \frac{\left\lvert u^\eps(x)-u^\eps(y)\right\rvert^{p-2}}{|x-y|^{1+sp}}\left(\sum_{k \in \Z}\left(\rho_\eps(y-x_j-kT)- \rho_\eps(x-x_j-kT)\right) \right) \rho_\eps(y-x_i) \,\ud x \ud y\\
=&p(p-1) \int_{\R} \int_{\R} \frac{\left\lvert u^\eps(x)-u^\eps(y)\right\rvert^{p-2} }{|x-y|^{1+sp}} \left(\sum_{k \in \Z} \left(\rho_\eps(x-x_j-kT)- \rho_\eps(y-x_j-kT)\right)\right) \\
& \hspace{10cm} \cdot \left(\rho_\eps(x-x_i) -\rho_\eps(y-x_i) \right)\ud y \ud x,
\end{align*}
where we have first divided the integral into two identical parts, exchanging the variables in the process; then we have added all the \(T\)-translated mollifiers for each variable (adding 0 since the supports are disjoint for \(k\neq0\)) to achieve symmetry and combined the two integrals together once again. 

Now, we observe that the expression above is well defined. Indeed, let us consider only the term with \(k=0\), since every other term in the sum is bounded:
\[
    I \coloneqq \int_{B_\eps(x_i)}\int_{B_\eps(x_j)} \frac{\left\lvert u^\eps(x)-u^\eps(y)\right\rvert^{p-2} }{|x-y|^{1+sp}}\left(\rho_\eps(x-x_j)- \rho_\eps(y-x_j)\right)\left(\rho_\eps(x-x_i) -\rho_\eps(y-x_i) \right)\ud y \ud x.
\]
Using a Taylor expansion around \(x\) with respect to \(y\) we get:
\[
    |u^\eps(x)-u^\eps(y)|^{p-2} = |(u^\eps)'(x)(y-x)+O(|x-y|^2)|^{p-2} = |(u^\eps)'(x)|^{p-2}|x-y|^{p-2} + O(|x-y|^{p-1}),
\]
if \((u^\eps)'(x)\neq 0\), otherwise the quantity above is of one order less, and 
\[
    \rho_\eps(x)-\rho_\eps(y) = \rho_\eps'(x)(y-x)+O(|x-y|^2).
\]
Substituting in I, we get
\[
    I \leq \int_{B_\eps(x_i)}\int_{B_\eps(x_j)} |(u^\eps)'(x)|^{p-2}\|\rho_\eps'\|_\infty^2|x-y|^{p-1-sp} + O(|x-y|^{p-sp}) \ud y \ud x
\]
that is finite for \(p-1-sp>-1\), since \(p(1-s)>0\), which is always verified. We must observe that when \(p\in(1,2)\) and \(u^\eps(x)=u^\eps(y)\) for some \(x\) and \(y\) distinct, then \(\left\lvert u^\eps(x)-u^\eps(y)\right\rvert^{p-2}= O(|x-y|)\). Once again \cref{eq: var seconda pura problematica} follows from \cref{eq: var seconda mista problematica,eq: somma zero hessiana}.

\end{proof}
\begin{remark} \label{rmk: var prima supercritica traslazioni orizz in 0}
    We observe that, for the equispaced configuration, the function \(u^\eps[\bold{X}^\mathrm{eq}]\) is \(1\)-periodic; therefore, \cref{eq: var prima supercritica s>0} is independent of the choice of \(x_i\). Hence, we can centre the first integral in \(x=0\)  and obtain
    \begin{equation*}
        \partial_{x_j} \F^{s,\eps}_p \left[\bold{X}^\mathrm{eq}\right]=\partial_{x_i} \F^{s,\eps}_p \left[\bold{X}^\mathrm{eq}\right] = 2p \int_{B_\eps} \int_{\R} \frac{ \left\lvert u^\eps(x)-u^\eps(y)\right\rvert^{p-2}(u^\eps(x)-u^\eps(y))}{|x-y|^{1+sp}} \,\ud y \ud x.
    \end{equation*}
\end{remark}
\begin{remark}
    As in the sub-critical case, we can define the matrix \(F^\eps\left[\bold{X} \right] \coloneqq (F^\eps_{ij}\left[\bold{X}\right])_{i,j}\), where \(F^{\eps}_{ij} \left[ \bold{X} \right] \coloneqq \partial^2_{x_i x_j}\F^{s,\eps}_p\left[\bold{X}\right]\). It is easy to check that, also in this setting, the sum of the elements in every row (or column) is zero.
\end{remark}

\begin{theorem}
Let \(\eps>0\), \(p\geq 1\), \(s\in(0,1)\) such that \(sp\geq 1\). Then, the configuration \(\bold{X}^\mathrm{eq}\) is the unique minimiser for the functional \(\F^{s,\eps}_p\) in the space \(\X_T^{4\eps}\).
\end{theorem}
\begin{proof}
We start by showing, as we did in the sub-critical case, that the first variation is zero when computed on the equispaced configuration \(\bold{X}^\mathrm{eq}\). Indeed, when calculated on the equispaced configuration, the function \(u^\eps[\bold{X}^\mathrm{eq}]\) is odd, and then the first variation is zero by symmetry on the integration domain (see \cref{rmk: var prima supercritica traslazioni orizz in 0}). The next step is showing that the matrix \(F^\eps\left[\bold{X} \right] \)  is positive definite. This follows from the same computation in \cref{eq: conto positività var seconda}, using \cref{rmk: somma zero}, which leads to 
\[
F^{\eps}[\bold{X}] \xi\cdot\xi = -\sum_{i=1}^T \sum_{j>i} F^\eps_{ij} \left(\xi_i-\xi_j\right)^2,
\]
for every \(\xi \in \R^T\). Moreover, we have that for every \(i,j \in \{1,\ldots,T\}\), where \(i \neq j\), it holds true from \cref{eq: var seconda mista supercritica s>0} that \( F^\eps_{ij}\leq 0\). This is where the hypothesis \(\ud[\bold X]\geq 4\eps\) comes into play.

\end{proof}

Some considerations are in order regarding the super-critical regime. Looking at \cref{eq: var seconda mista problematica}, when jump points are less than \(4\eps\) apart, the overlapping interactions of the convolution kernels prevent us from easily determining the sign of the second variation. Indeed, while the interaction of points outside the overlap region yields a negative contribution, the points falling within the intersection of the mollifiers' supports generate a positive one. As a result, a quantitative estimate establishing the sign of the matrix \(F^\eps\) for closely packed points is not easily achievable.

Furthermore, perturbative arguments based on the first variation, such as the one in \Cref{lemma: cuspidi}, inevitably fail in this setting. Due to the $\eps$-regularisation, configurations with overlapping jumps no longer behave as isolated cusps of the energy; instead, they become critical points, meaning the energy locally loses its strict convexity.

Finally, extending the strategy of \Cref{lemma: cuspidi p=1} to the super-critical case is far from computationally trivial. In principle, one would like to prove that configurations with a fixed minimal distance are always energetically favourable compared to those where the minimal distance shrinks with \(\eps\). However, in the super-critical case, the main obstruction is that \virgolette{core} interactions, i.e. those between points at a distance of order \(\eps\), exhibit a blow-up that is not of lower order compared to macroscopic interactions.

\subsection{Some further results for the critical case \(sp=1\)}

In the critical regime, one can observe that core interactions are negligible. Provided that, we are able to prove that the energy of a configuration whose minimal distance \(\delta(\eps) \to 0\) as \(\eps \to 0\) grows like \(-\log\eps\).

\begin{proposition} \label{prop: caso critico salvo}
	Let \(s\in(0,1)\), \(p> 1\) such that \(sp= 1\), \(\delta>0\) and \(R>\delta\) fixed. Then
	\begin{equation} \label{eq: stima liminf caso critico}
	\liminf_{\eps\to 0} \inf_{\substack{\bold X \in \X_T\\ \ud[\bold X]\leq \delta}}\left(\F^{s,\eps}_p[\bold X] + (2^{2-p})T\log\eps \right)\geq -4(1-2^{1-p}) \log\delta + C(p,T,R).
	\end{equation}
\end{proposition}

\begin{proof} Defining \(\mu_\eps^{\bold X} \coloneqq \mu^{\bold X} * \rho_\eps\) and applying the change of variables \(y\mapsto x+h\) to the energy we write
\begin{align}
		\F^{s,\eps}_p[\bold X] &= \int_0^T\int_\R \frac{|u^\eps(x)-u^\eps(x+h)|^p}{|h|^2} \,\ud h \ud x \nonumber\\
		&= 2\int_0^{+\infty} \frac{1}{h^2}\int_0^T |u^\eps(x+h)-u^\eps(x)|^p \ud x \ud h \nonumber\\
		&\geq 2\int_0^{R} \frac{1}{h^2}\int_0^T |\mu_\eps^{\bold X}([x,x+h])-h|^p \ud x \ud h \nonumber\\
		& \geq 2\int_0^{R} \frac{1}{h^2}\int_0^T 2^{1-p}|\mu_\eps^{\bold X}([x,x+h])|^p - h^p \ud x \ud h \nonumber\\
		&= 2^{2-p}\int_0^{R} \frac{1}{h^2}\int_0^T |\mu_\eps^{\bold X}([x,x+h])|^p \ud x \ud h - \frac{2R^{p-1}T}{p-1}, \label{eq: energia con misura mollificata}
\end{align}
where we have used the inequality \(|a-b|^p\geq 2^{1-p}|a|^p-|b|^p\), true for \(p\geq 1\). Observe that, using Jensen inequality on the convolution integral, and using the absolute continuity of \(\mu_\eps^{\bold X}\) we have
\begin{equation} \label{eq: eh voleeeevi}
        \int_0^{2\eps}\frac{1}{h^2}\int_0^T |\mu_\eps^{\bold X}([x,x+h])|^p \ud x \ud h  \leq \int_0^{2\eps} C|h|^{p-2} \ud h = O(\eps^{p-1}).
\end{equation}
Therefore, without loss of generality, let \(h>2\eps\) and take \(\eta\in(-\eps,\eps)\), then 
\[
	[x-\eta,x+h-\eta] \supset [x+\eps,x+h-\eps]
\]
and, since the measures \(\mu_\eps^{\bold X}\) and \(\mu^{\bold X}\) are positive, we can compute
\begin{align}
	\mu_\eps^{\bold X}([x,x+h]) &= \int_x^{x+h} \mu_\eps^{\bold X}(y) \ud y = \int_{-\eps}^{\eps} \mu^{\bold X}([x-\eta,x+h-\eta]) \rho^\eps(\eta) \ud \eta \label{eq: misura demollificata} \\
	&\geq  \int_{-\eps}^{\eps}\mu^{\bold X}([x+\eps,x+h-\eps]) \rho^\eps(\eta) \ud \eta \nonumber\\
	&=  \mu^{\bold X}([x+\eps,x+h-\eps]).\nonumber
\end{align}
By strict monotonicity of \(z\mapsto z^p\) for \(z>0\), substituting \cref{eq: misura demollificata} into \cref{eq: energia con misura mollificata} yields 
\begin{equation} \label{eq: stima misure discrete}
	\F^{s,\eps}_p[\bold X]  \geq 2^{2-p} \int_{2\eps}^R \frac{1}{h^2}\int_0^T |\mu^{\bold X}([x+\eps,x+h-\eps])|^p \ud x \ud h - \frac{2R^{p-1}T}{p-1}.
\end{equation}

\begin{figure}[h]
	\centering
	\begin{tikzpicture}
        \draw[line width=4pt, black!20] (0,0) -- (.5,0);
        \draw[line width=4pt, black!20] (1,0) -- (2.5,0);
        \draw[line width=4pt, black!20] (4.5,0) -- (5.5,0);
        \node at (1.5,0) [above]{\footnotesize \(A_t(1)\)};
        
		\draw[->] (-1,0) -- (7,0);
		\draw[->] (0,-1) -- (0,4);
		\node at (7,0) [below]{\(x\)};
		\node at (0,4) [left]{\(\mu^{\bold X}([x,x+t])\)};
		\draw[thick, -{stealth}] (0,-.5) -- (1,-.5) node[below=-1pt]{t};
		\node at (0,1) [left]{\(1\)};
		\node at (0,2) [left]{\(2\)};
		\node at (0,3) [left]{\(3\)};
		
		\draw[dashed,black!30] (.5,1) -- (.5,0);
		\draw[dashed,black!30] (1,1) -- (1,0);
		\draw[dashed,black!30] (2.5,3) -- (2.5,1);
		\draw[dashed,black!30] (3,3) -- (3,2);
		\draw[dashed,black!30] (3.5,2) -- (3.5,0);
		\draw[dashed,black!30] (4.5,1) -- (4.5,0);
		\draw[dashed,black!30] (5.5,1) -- (5.5,0);
		

		\draw (.5,-.1) -- (.5,.1);
		\node at (.5,0) [below]{\(x_1\)};
		\draw (2,-.1) -- (2,.1);
		\node at (2,0) [below]{\(x_2\)};
		\draw (3,-.1) -- (3,.1);
		\node at (3,0) [below]{\(x_3\)};
		\draw (3.5,-.1) -- (3.5,.1);
		\node at (3.55,0) [below]{\(x_4\)};
		\node at (3.55,-.25) [below]{\(x_5\)};
		\draw (5.5,-.1) -- (5.5,.1);
		\node at (5.5,0) [below]{\(x_6\)};
		
		\draw[thick] (0,1) -- (.5,1);
		\filldraw (.5,1) circle (1pt);
		\draw[thick] (.5,0) -- (1,0);
		\filldraw[fill=white] (.5,0) circle (1pt);
		\draw[thick] (1,1) -- (2.5,1);
		\filldraw[fill=white] (1,0) circle (1pt);
		\filldraw (1,1) circle (1pt);
		\filldraw (2,2) circle (1pt);
		\filldraw[fill=white] (2,1) circle (1pt);
		\draw[thick] (2.5,3) -- (3,3);
		\filldraw[fill=white] (2.5,1) circle (1pt);
		\filldraw (2.5,3) circle (1pt);
		\draw[thick] (3,2) -- (3.5,2);
		\filldraw (3,3) circle (1pt);
		\filldraw[fill=white] (3,2) circle (1pt);
		\filldraw (3.5,2) circle (1pt);
		\draw[thick] (3.5,0) -- (4.5,0);
		\filldraw[fill=white] (3.5,0) circle (1pt);
		
		\draw[thick] (4.5,1) -- (5.5,1);
		\filldraw[fill=white] (4.5,0) circle (1pt);
		\filldraw (4.5,1) circle (1pt);
		
		\draw[thick] (5.5,0) -- (6,0);
		\filldraw[fill=white] (5.5,0) circle (1pt);
		\filldraw (5.5,1) circle (1pt);

        \draw[thick,dashed] (6,0) -- (6.8,0);

	\end{tikzpicture}
	\label{fig: misura che conta i salti}
	\caption{An example of how the measure \(\mu^{\bold X}\) behaves for a particular choice of \(\bold X\) and a fixed \(t\). For convenience the set \(A_t(1)\) has been highlighted.}
\end{figure}
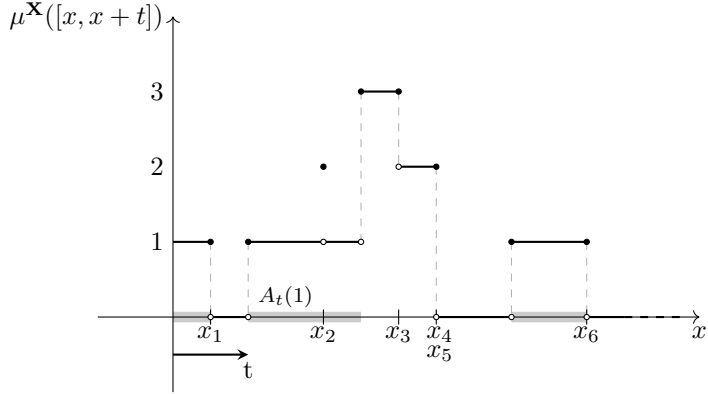%
\noindent
Now, by taking 
\[
	A_t(k) \coloneqq \{x\in[0,T)\, |\, \mu^{\bold X}([x,x+t])=k\},
\]
we apply the change of variables \(y=x+\eps\) and \(t=h-2\eps\) in \cref{eq: stima misure discrete}, using periodicity of \(\mu^{\bold X}\), yielding 
\begin{align*}
		\F^{s,\eps}_p[\bold X] &\geq 2^{2-p} \int_{0}^{R-2\eps} \frac{1}{(t+2\eps)^2}\int_0^T |\mu^{\bold X}([y,y+t])|^p \ud y \ud t - \frac{2R^{p-1}T}{p-1}\\
		& = 2^{2-p} \int_{0}^{R-2\eps} \frac{1}{(t+2\eps)^2} \sum_{k=1}^{+\infty}k^p|A_t(k)|  \ud t - \frac{2R^{p-1}T}{p-1}\\
		&\geq 2^{2-p} \int_{0}^{R-2\eps} \frac{1}{(t+2\eps)^2} \left( \sum_{k=1}^{+\infty}k|A_t(k)| + (2^p-2)\sum_{k=2}^{+\infty}|A_t(k)| \right) \ud t - \frac{2R^{p-1}T}{p-1},
\end{align*}
where we used the inequality \(k^p \geq k + (2^p-2)\rchi_{[2,+\infty)}(k)\) for \(k\in\N\) and \(p\geq 1\). Now observe that we can write 
\[
	\sum_{k=2}^{+\infty}|A_t(k)| = |\{x\in [0,T) \,|\, \mu^{\bold X}([x,x+t])\geq 2\}| \geq (t-\delta_{\bold X})_+,
\]
where we denote \(0\leq \delta_{\bold X} \coloneqq \ud[\bold X]\ \leq \delta\). By definition of \(\delta_{\bold X}\), if we take \(t>\delta_{\bold X}\), there exists at least an interval (of size \(t-\delta_{\bold X}\)) contained in the set above. Observe also that 
\[
	\sum_{k=1}^{+\infty}k|A_t(k)| = \int_0^T \mu^{\bold X}([y,y+t]) \ud y = \int_0^T \sum_{x_i \in \bold X} \rchi_{[y,y+t]}(x_i) \ud y = \sum_{x_i \in \bold X} \int_0^T \rchi_{[x_i-t, x_i]}(y) \ud y = \sum_{x_i \in \bold X} t = Tt
\]
indeed the measure \(\mu^{\bold X}\) counts exactly \(T\) jumps in a period. Putting everything together we get 
\begin{align*}	
		\F^{s,\eps}_p[\bold X] &\geq 2^{2-p} \int_{0}^{R-2\eps} \frac{Tt}{(t+2\eps)^2}\, \ud t + 4(1-2^{1-p})\int_{0}^{R-2\eps} \frac{(t-\delta_{\bold X})_+}{(t+2\eps)^2}\,  \ud t- \frac{2R^{p-1}T}{p-1} \\
		&= 2^{2-p}T \int_{2\eps}^{R} \frac{h-2\eps}{h^2} \,\ud h + 4(1-2^{1-p})\int_{2\eps}^{R} \frac{(h-2\eps-\delta_{\bold X})_+}{h^2}\, \ud h- \frac{2R^{p-1}T}{p-1} \\
		&\geq -2^{2-p}T \log\eps - 4(1-2^{1-p}) \log(2\eps+\delta) + C(p,T,R).
\end{align*}

\end{proof}

\begin{remark}
	As anticipated, the strategy employed in the proof of \Cref{prop: caso critico salvo} fails when generalised to the super-critical regime (\(sp > 1\)). Looking at \cref{eq: eh voleeeevi}, we can observe that for \(sp > 1\) the integral no longer yields a vanishing contribution as \(\eps \to 0\). 
	
	The underlying reason is structural. In the critical case (\(sp=1\)), the energy diverges logarithmically. This allows us to isolate and discard the core interactions (distances \(h < 2\eps\)). On one hand, as shown in \cref{eq: eh voleeeevi}, the actual energy contribution of this core region goes to zero as \(\eps \to 0\). On the other hand, the specific choice of the cut-off boundary \(h=2\eps\) is the right one. Indeed, take a general cut-off point \(\delta_\eps\) going to zero as \(\eps\) goes to zero. Then, if \(\delta_\eps = o(\eps)\) the estimate with \(h<2\eps\) still holds. On the other hand if \(\frac{\delta_\eps}{\eps}\to +\infty\) this estimate does not lead to the optimal constant for the logarithmic term. 
	
	Conversely, in the super-critical case (\(sp > 1\)), the energy exhibits a polynomial blow-up of order \(\eps^{1-sp}\). Because of this, the integral over the core \(h < 2\eps\) does not vanish. Instead, it scales with the exact same power of \(\eps\) as the rest of the integral. Thus, the core contribution cannot be ignored or absorbed into a lower-order term. Its value heavily depends on the specific profile of the mollifier \(\rho\) and the chosen cut-off distance, preventing the extraction of an optimal leading-order constant as we did in \cref{eq: stima liminf caso critico}.
\end{remark}

\subsection{Convexity argument and minimisers for \(\F^0_p\)} 
We now show that the equispaced configuration is the unique minimiser for \(\F^0_p\) within the class \(\mathbb{X}_T\). To this end, we employ a convexity argument, inspired by \cite{Müller93} for \(p=2\), based on the lengths of the affine segments between the jumps, and later expanded upon by \cite{RenWei03} for a more general functional, still in the case \(p=2\).

Let us recall the  \(\Gamma\)-limit of \(s\F^s_p\) as \(s \to 0^+\) in \(d=1\):
\[
    \F^0_p(u)= \int_0^T\int_0^T |u(x)-u(y)|^p  \ud y \ud x,
\]
where we recall \(T \in \N\). We observe that, given \(\bold X \in \X_T\), we can write the energy as depending only on the length of the intervals between the jump points. To facilitate the use of symmetry arguments, it is convenient to centre the integration domains at the origin
\begin{equation*}
\F^0_p(u[\bold X])= \sum_{i=1}^{T} \sum_{j=1}^T \Ff(L_i,L_j,c_{ij})\coloneqq \sum_{i=1}^{T} \sum_{j=1}^T \int_{-L_i}^{L_i} \int_{-L_j}^{L_j} |x-y-c_{ij}|^p \ud y \ud x,
\end{equation*}
where \(L_i,\,L_j\in\R^+\) and \(c_{ij}\in\R\) are chosen such that the following equality holds
\[
\Ff(L_i,L_j,c_{ij}) =\int_{x_{i-1}}^{x_{i}} \int_{x_{j-1}}^{x_{j}} |u(x)-u(y)|^p \ud y \ud x.
\]
From now on, we are going to identify \(\F^0_p\left[\bold{L}\right] \coloneqq\F^0_p\left[\bold{X}\right] \coloneqq\F^0_p(u[\bold{X}])\), where \(\bold{L}=\left(L_1,\ldots,L_T\right) \in \R^T\), with each \(L_i\) defined as \(2 L_i\coloneqq |x_{i+1}-x_i|\), which satisfies the constraint \(2\sum_{i=1}^{T}L_i=T\) for every configuration.

Let us observe that, in the equality above, the quantity \(c_{ij}\) is uniquely determined by \(L_i\) and \(L_j\). Nevertheless, it will be useful in the upcoming computations to think of \(\Ff\) as a function of three variables we can manipulate independently.

Before stating the final result, we give a preliminary Lemma.

\begin{lemma} \label{lemma: min per c=0}
For every \(p\geq 1\), \(\ell,\ell' \in [0,+\infty)\) and \(c \in \R\), the function \(c \mapsto \Ff(\ell,\ell',c)\) has the unique minimiser \(c=0\).
\end{lemma}
\begin{proof}
Indeed, we have
\[
    \left.\frac{\ud}{\ud c} \Ff(\ell, \ell',c) \right|_{c=0} =-p \int_{-\ell}^{\ell} \int_{-\ell'}^{\ell'} (x-y) |x-y|^{p-2}\ud y \ud x =0
\]
due to its symmetry; moreover,
\begin{align*}
    \frac{\ud^2}{\ud c^2} \Ff(\ell, \ell',c)& =p(p-1)\int_{-\ell}^{\ell} \int_{-\ell'}^{\ell'} |x-y-c|^{p-2} \ud y \ud x \geq 0\,,
\end{align*}
for every \(c \in \R\).

\end{proof}

\begin{theorem} \label{thm: sistenza minimo per s=0}
	Let \(p\geq 1\). The configuration \(\bold{X}^\mathrm{eq}\) is (up to translations) the unique minimiser for the functional \(\F^0_p\) in the space \(\X_T\).
\end{theorem}

\begin{proof}
Denoting with \(\Fff(L_i, L_j)\coloneqq \Ff(L_i, L_j,0)\) and using \Cref{lemma: min per c=0} we have
\begin{equation*}
    \F^0_p \left[\bold{L}\right] =\sum_{i=1}^T \sum_{j=1}^T \Ff(L_i, L_j,c_{ij}) \geq \sum_{i=1}^T\sum_{j=1}^T \Fff(L_i,L_j) \eqqcolon \FF\left[\bold{L}\right].
\end{equation*}
Firstly, we prove convexity of the function 
\[
	\Fff(\ell,\ell') = \frac{2}{(p+1)(p+2)} \left[ (\ell+\ell')^{p+2}-\left|\ell-\ell'\right|^{p+2}\right]
\] 
with respect to the two variables. Indeed, taking the derivative with respect to \(\ell\) (and by symmetry the same holds for \(\ell'\)) we get
\[
    \frac{\mathrm{d}}{\mathrm{d}\ell}\Fff(\ell,\ell')=\frac{2}{p+1}\left[(\ell+\ell')^{p+1}-(\ell-\ell')|\ell-\ell'|^p\right]
\]
and
\[
    \frac{\mathrm{d}^2}{\mathrm{d}\ell^2}\Fff(\ell,\ell')=2\left[(\ell+\ell')^p-\left|\ell-\ell'\right|^{p}\right].
\]
Both derivatives are continuous and the second derivative is strictly positive by virtue of the inequality \(\ell+\ell'>\left|\ell-\ell'\right|\) true for \(\ell,\ell'>0\). 
Using convexity of \(\Fff\) and Jensen inequality we can make the following considerations on \(\FF\):
\begin{align*}
    \frac{1}{T^2}\FF[\bold L] &= \frac{1}{T} \sum_{i=1}^T \frac{1}{T}\sum_{j=1}^T \Fff(L_i,L_j) \geq \frac{1}{T} \sum_{i=1}^T \Fff\left(L_i,\frac{1}{T}\sum_{j=1}^T L_j\right) \\
    & = \frac{1}{T} \sum_{i=1}^T \Fff\left(L_i,\frac{1}{2}\right) \geq \Fff\left(\frac{1}{T} \sum_{i=1}^T L_i,\frac{1}{2}\right) \\
    &= \Fff\left(\frac{1}{2},\frac{1}{2}\right) = \frac{1}{T^2} \sum_{i=1}^T\sum_{j=1}^T\Fff\left(\frac{1}{2},\frac{1}{2}\right) = \frac{1}{T^2} \FF[\bold{L}^\mathrm{eq}],
\end{align*}
where \(\bold L^\mathrm{eq}\) is the vector of lengths associated to the configuration \(\bold X^\mathrm{eq}\).
This proves that \(\bold L^{\mathrm{eq}}\) minimises \(\FF\), and since the minimum of \(\Ff(L_i,L_j,c_{ij})\) is achieved when all \(c_{ij}=0\), which is precisely the case for \(\bold X^{\mathrm{eq}}\), the claim is proven.
\end{proof}

\begingroup
\renewcommand{\addcontentsline}[3]{}
\section*{Acknowledgments}
\noindent We wish to express our deepest gratitude to our PhD supervisor Marcello Ponsiglione and Dr. Lucia De Luca, for their constant guidance, insightful discussions, and invaluable support during the preparation of this manuscript.
\endgroup

\bibliographystyle{plain}
\bibliography{Bibliografia}

\par\vspace{2\baselineskip}
\small


(G. Pini) \textsc{Dip. di Matematica, Univ. Roma-I \virgolette{La Sapienza}, Piazzale Aldo Moro 5, 00185, Roma, Italy} \\
\textit{E-mail address, G. Pini:} \texttt{giovanni.pini@uniroma1.it}

\par\medskip 

(F. Santilli) \textsc{Dip. di Matematica, Univ. Roma-I \virgolette{La Sapienza}, Piazzale Aldo Moro 5, 00185, Roma, Italy} \\
\textit{E-mail address, F. Santilli:} \texttt{francesco.santilli@uniroma1.it}

\end{document}